\documentclass{article} % For LaTeX2e
\usepackage{iclr2024_conference,times}

\usepackage{amsmath,amsfonts,bm}

\def\eqref#1{equation~\ref{#1}}
\def\1{\bm{1}}

\DeclareMathAlphabet{\mathsfit}{\encodingdefault}{\sfdefault}{m}{sl}
\SetMathAlphabet{\mathsfit}{bold}{\encodingdefault}{\sfdefault}{bx}{n}

\usepackage{amsfonts,amssymb,amsbsy,latexsym,tabulary,graphicx,times,caption,fancyhdr,amsthm}
\usepackage{setspace}
\usepackage[utf8]{inputenc} % allow utf-8 input
\usepackage[T1]{fontenc}    % use 8-bit T1 fonts
\usepackage{booktabs}       % professional-quality tables
\usepackage{amsfonts}       % blackboard math symbols
\usepackage{nicefrac}       % compact symbols for 1/2, etc.
\usepackage{microtype}      % microtypography
\usepackage{xcolor}         % colorsf
\usepackage[T1]{fontenc}
\usepackage{bm}
\usepackage{lineno}
\usepackage{hyperref}
\usepackage{mathrsfs}
\usepackage{array}
\usepackage{multirow}
\usepackage{algorithm}
\usepackage{algorithmic}
\usepackage{hyperref}
\usepackage{float}
\usepackage{subfigure}
\usepackage[justification=centering]{caption}
\usepackage{marvosym}
\usepackage{verbatim}
\usepackage{enumitem}
\usepackage{tablefootnote}
\usepackage{caption}
\usepackage{subcaption}
\usepackage{ragged2e}

\numberwithin{equation}{section}
\newtheorem{thm}{Theorem}

\newtheorem{rmk}{Remark}

\renewcommand{\eqref}[1]{\ref{#1}}
\newcommand{\tensor}[1]{\bm{\mathcal{#1}}}

\DeclareGraphicsExtensions{.pdf,.jpeg,.jpg,.png,.eps}
\graphicspath{{./figures/},{./myfigures/}}

\usepackage{url}            % simple URL typesetting
\usepackage{nicefrac}       % compact symbols for 1/2, etc.
\title{A Novel Block-Alternating Iterative Algorithm for Retrieving Top-$k$ Elements from Factorized Tensors}

\author{Chuanfu~Xiao,
Jiaxin~Zeng* \thanks{ Use footnote for providing further information
about author (webpage, alternative address)---\emph{not} for acknowledging
funding agencies.  Funding acknowledgements go at the end of the paper.} \\
Department of Computer Science\\
Cranberry-Lemon University\\
Pittsburgh, PA 15213, USA \\
\texttt{\{hippo,brain,jen\}@cs.cranberry-lemon.edu} \\
\And
Ji Q. Ren \& Yevgeny LeNet \\
Department of Computational Neuroscience \\
University of the Witwatersrand \\
Joburg, South Africa \\
\texttt{\{robot,net\}@wits.ac.za} \\
\AND
Coauthor \\
Affiliation \\
Address \\
\texttt{email}
}

\begin{document}

\maketitle

\begin{abstract}
Tensors, especially higher-order tensors, are typically represented in low-rank formats to preserve the main information of the high-dimensional data while saving memory space. In practice, only a small fraction elements in high-dimensional data are of interest, such as the $k$ largest or smallest elements. Thus, retrieving the $k$ largest/smallest elements from a low-rank tensor is a fundamental and important task in a wide variety of applications. In this paper, we first model the top-$k$ elements retrieval problem to a continuous constrained optimization problem. To address the equivalent optimization problem, we develop a block-alternating iterative algorithm that decomposes the original problem into a sequence of small-scale subproblems. Leveraging the separable summation structure of the objective function, a heuristic algorithm is proposed to solve these subproblems in an alternating manner. Numerical experiments with tensors from synthetic and real-world applications demonstrate that the proposed algorithm outperforms existing methods in terms of accuracy and stability.
\end{abstract}

\section{Introduction}

Large volumes of high-dimensional data, such as simulation data, hyperspectral images, video, have sprung up in scientific computing, remote sensing, computer vision, and many other applications. Typically, these high-dimensional data can be naturally represented by tensors, and are stored in a lossy compression manner via tensor decomposition representations, such as CANDECOMP/PARAFAC (CP) \cite{hitchcock1927expression}, Tucker \cite{tucker1966some}, and tensor train (TT) \cite{oseledets2011tensor} to mitigate the curse of dimensionality \cite{kolda2009tensor,cong2015tensor,sidiropoulos2017tensor,cichocki2016tensor,cichocki2017tensor,zhu2024aptt}. In practical scenarios, many applications rely solely on a subset of high-dimensional data, corresponding to a few elements of a tensor. For example, in recommendation systems, {the $k$ largest elements correspond to the most meaningful concerns for personalized recommendations}~\cite{symeonidis2016matrix,frolov2017tensor,zhang2021dynamic}. 
In quantum simulations, quantum states are usually compressed in tensor decomposition representations to reduce memory usage. Meanwhile, the measurement of quantum states can be simulated using the maximum likelihood estimation, which is equivalent to retrieving the largest magnitude element in a factorized tensor \cite{zhou2020limits,yuan2021quantum,ma2022low}. Moreover, numerous practical tasks can also be modeled as the top-$k$ elements retrieval problem~\cite{sozykin2022ttopt,sidiropoulos2023minimizing}.

In recent years, there is a series of works have been conducted on the top-$k$ elements retrieval problem from a factorized tensor (i.e., high-dimensional data represented by tensor decompositions) \cite{sozykin2022ttopt,batsheva2023protes,sidiropoulos2023minimizing,shetty2024tensor,zurek2025tensor,sozykin2025high}. For tensors in CP format, a method called star sampling~\cite{lu2017sampling} was proposed to find the $k$ largest elements, which is a generalization of the diamond sampling method proposed in \cite{ballard2015diamond} for the matrix case. Sampling methods applicable to other tensor decomposition formats have also been developed~\cite{sozykin2022ttopt,chertkov2023tensor,ryzhakov2024black}.
% TT format \cite{sozykin2022ttopt,chertkov2023tensor} and hierarchical Tucker format \cite{ryzhakov2024black}. 
However, the accuracy of sampling methods is highly sensitive to sample size and quality, which in turn depend on the underlying structure of factorized tensors. This dependency will lead to unstable performance. Moreover, such methods are limited to identifying the $k>1$ largest elements and cannot be directly generalized to retrieving the $k$ smallest elements or others. Another more efficient approach is to formulate the largest/smallest element retrieval problem in a factorized tensor as a continuous optimization problem. As demonstrated in~\cite{espig2013efficient,espig2020iterative}, this problem can be transformed into a symmetric eigenvalue problem, where the symmetric matrix is a diagonal matrix consisting of all entries of the factorized tensor. Consequently, the classical power iteration (inverse iteration) method can be employed to find the largest (smallest) element~\cite{espig2013efficient,espig2020iterative,soley2021iterative}. It is worth noting that the Hadamard product (Hadamard inverse) operation involved in power iteration (inverse iteration) method leads to rapid growth in tensor rank. To address this issue, a recompression operation needs to be introduced \cite{oseledets2011tensor,al2022parallel,al2023randomized,sun2024hatt}. Not only that, the accumulated error from recompression during the iterative process of the power iteration (inverse iteration) method 
may fail in locating the location corresponding to the largest/smallest element\footnote{To complement this approach, we design a post-processing strategy that identifies the location of the largest/smallest elements by computing the rank-one approximation of the tensor corresponding to the eigenvector \cite{regalia2000higher,de2000best,zhang2001rank,xiao2024hoscf}.}. 
In~\cite{sidiropoulos2017tensor}, the authors modeled the top-$1$ retrieval problem as an equivalent continuous constraint optimization problem and proposed the projected gradient descent (PGD) method to solve it~\cite{sidiropoulos2023minimizing}. Unfortunately, the PGD method exhibits significant sensitivity to hyperparameter selection (e.g., the step size), resulting in unstable performance across different tasks. Notably, existing algorithms derived from the continuous optimization perspective cannot be directly applied to the case of $k>1$.

For the top-$k$ elements retrieval problem in a factorized tensor, we follow the symmetric eigenvalue model given in \cite{espig2013efficient,espig2020iterative}. Within this model, we observe that the tensor corresponding to the eigenvector has a rank-one structure. Based on this, we then present an equivalent continuous constrained optimization reformulation in this paper. To solve the continuous constrained optimization problem, we design a novel block-alternating iterative algorithm. Specifically, this algorithm transforms the original large-scale problem into a series of solvable small-scale subproblems, and a heuristic method is further developed to efficiently solve these subproblems based on the separable summation form of the objective function. 
Numerical experiments with factorized tensors from synthetic and real-world applications demonstrate that the proposed block-iterative algorithm enjoys the advantages in (1) \emph{versatility}: it can retrieve the~$k$ largest/smallest elements and their locations; and (2) \emph{stability}: it achieves a significant improvement in accuracy on factorized tensors with different distributions.

Our main contributions consist of three parts, summarized as follows:
\begin{itemize}
\item \emph{Model}: Inspired by~\cite{espig2013efficient,espig2020iterative}, we model the top-$k$ elements retrieval problem as a continuous constrained optimization problem based on the rank-one structure of the tensor corresponding to the eigenvector.

\item \emph{Algorithm}: We propose a novel block-alternating iterative algorithm to solve the equivalent continuous constrained optimization problem. 

\item \emph{Application}: We apply the proposed block-alternating iterative algorithm to the measurement phase in quantum circuit simulations. Numerical results illustrate that the proposed algorithm outperforms existing algorithms in terms of accuracy.
\end{itemize}

The rest of the paper is organized as follows. Section \ref{sec:prob} introduces the top-$k$ elements retrieval problem and presents its equivalent continuous constrained optimization reformulation. Section \ref{sec:dmrg} shows the framework of the block-alternating iterative algorithm for the continuous constrained optimization problem. Numerical results on synthetic and real-world tensors are reported in Section \ref{sec:numexp}. The conclusion of this paper is in Section \ref{sec:con}.

\textbf{Notations:} In this paper, we use boldface lowercase and capital letters (e.g., $\bm{a}$ and $\bm{A}$) to denote vectors and matrices. The boldface Euler script letter is used to denote higher-order tensors, e.g., a $d$th-order tensor can be expressed as $\bm{\mathcal{A}} \in\mathbb{R}^{n_{1}\times n_{2}\times\cdots\times n_{d}}$, where $n_p$ denotes the {dimension} of mode-$p$, and the $(i_{1},i_{2},\ldots,i_{d})$th element of it is represented by $\bm{\mathcal{A}}\left({i_{1},i_{2},\ldots,i_{d}}\right)$.

\section{Problem Formulation}\label{sec:prob}

In the paper, we only consider CP tensors. Simultaneously, the proposed model and algorithm can be naturally applied to other tensor decomposition formats, including but not limited to Tucker and TT tensors. Let $\bm{\mathcal{A}}\in\mathbb{R}^{n_{1}\times n_{2}\times\cdots\times n_{d}}$ be a $d$th-order tensor. Its CP format represents it as a sum of $R$ rank-one tensors, that is,
\setlength\abovedisplayskip{0.2cm}
\setlength\belowdisplayskip{0.2cm}
\begin{equation}\label{eq:cp}    \tensor{A}\mathrel{:=}\sum\limits_{r=1}^{R}\bm{U}_1(:,r)\circ\bm{U}_2(:,r)\circ\cdots\circ\bm{U}_{d}(:,r),
\end{equation}where ``$\circ$'' denotes the tensor outer product, $\left\{\bm{U}_p\in\mathbb{R}^{n_p\times R}:p=1,2\ldots,d\right\}$ are CP factors, and $R$ denotes the CP rank of $\tensor{A}$ if Eq. \eqref{eq:cp} is the sum of the smallest number of rank-one tensors \cite{hitchcock1927expression,kruskal1977three}. Specifically, a $d$th-order tensor $\tensor{A}$ in Eq. \eqref{eq:cp} is termed a CP tensor. For clarity, Fig.~\ref{fig:cp} depicts the diagram of a third-order CP tensor.
\begin{figure}[t]
\centering
\vspace{0.1cm}
\setlength{\abovecaptionskip}{0.1cm}  
\setlength{\belowcaptionskip}{0.1cm}  
\includegraphics[width=0.78\textwidth]{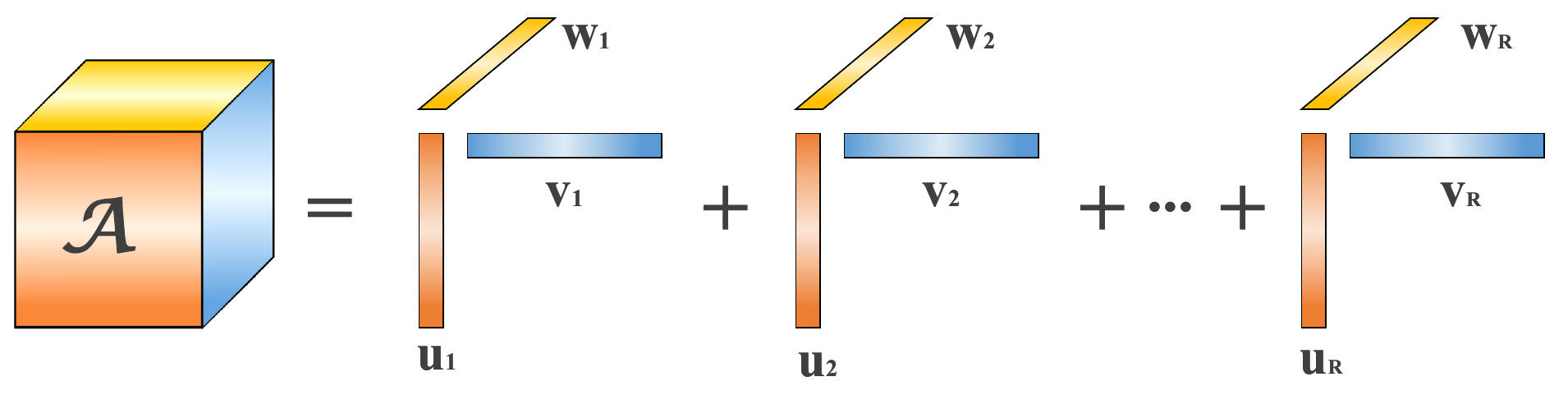}
\centering
\caption{The diagram of a third-order CP tensor.}
\label{fig:cp}
\end{figure}

The goal of this paper is to retrieve the top-$k$ elements of the CP tensor $\tensor{A}$, e.g., the $k$ largest or smallest elements and their corresponding locations. 
Without loss of generality, we take the $k$ largest elements retrieval problem as an example. Meanwhile, a similar result can be derived for retrieving the $k$ smallest elements. 
From \cite{espig2013efficient,espig2020iterative}, we know that retrieving the $k$ largest elements of the CP tensor $\tensor{A}$ can be modeled as a symmetric eigenvalue problem. More specifically, {the top-$k$ retrieval problem} is equivalent to solving the $k$ largest eigenpairs of a diagonal matrix $\bm{A}\in\mathbb{R}^{\prod\limits_{p=1}^dn_p\times\prod\limits_{p=1}^dn_p}$ that consists of all entries of $\tensor{A}$, i.e., 
\setlength\abovedisplayskip{0.2cm}
\setlength\belowdisplayskip{0.2cm}
\begin{equation}\label{eq:diagonal}
    \bm{A}(i,i) = \tensor{A}(i_1,i_2,\ldots,i_d)\quad \text{with}\quad i=i_1+\sum\limits_{q=1}^{d-1}(i_{q+1}-1)n_{1:q},
\end{equation}
where $n_{1:q}=\prod\limits_{p=1}^qn_p$. Further, the symmetric eigenvalue problem can be formulated as the following orthogonal-constrained optimization problem, which is
\begin{equation}\label{prob:symeigen}
    \max\limits_{\bm{X}\in\mathbb{R}^{\prod\limits_{p=1}^dn_p\times k}}\texttt{tr}\left(\bm{X}^{T}\bm{A}\bm{X}\right)\quad \text{s.t.}\quad \bm{X}^{T}\bm{X} = \bm{I}_{k},
\end{equation}
where $\texttt{tr}(\cdot)$ represents the matrix trace operator, and $\bm{I}_k$ is the $k\times k$ identity matrix. 

For convenience, we denote the Hadamard product and inner product of tensors as ``$*$'' and $\langle\cdot,\cdot\rangle$, respectively. In tensor representations, the multiplication of the diagonal matrix $\bm{A}$ and vector $\bm{X}(:,j)$ can be rewritten as the Hadamard product of $\tensor{A}$ and $\tensor{X}_j$, i.e., $\tensor{A}*\tensor{X}_j$, and the inner product of vectors $\bm{X}(:,j)$ and $\bm{A}\bm{X}(:,j)$ can also be rewritten as the inner product of tensors $\tensor{X}_j$ and $\tensor{A}*\tensor{X}_j$, i.e., $\langle\tensor{X}_j,\tensor{A}*\tensor{X}_j\rangle$, where $\tensor{X}_j\in\mathbb{R}^{n_1\times n_2\times\cdots\times n_d}$ is tensorized by $\bm{X}(:,j)$.
Thus, the optimization problem \eqref{prob:symeigen} can be reformulated as
\begin{equation}\label{prob:symeigen-t}
\begin{array}{l}
\max\limits_{\tensor{X}_j\in\mathbb{R}^{n_1\times n_2\times\cdots\times n_d}}\sum\limits_{j=1}^k\langle\tensor{X}_j,\tensor{A}*\tensor{X}_j\rangle, \\
      \text{s.t.}\quad\langle\tensor{X}_{i},\tensor{X}_j\rangle=\left\{\begin{array}{cc}
          1, & i=j,\\
          0, & i\neq j.
    \end{array}\right.
\end{array}
\end{equation}

%%%%%%%%%%%%%%%%%%%%%%%%%%%%%%%%
\begin{figure*}[t]
\centering
\vspace{0.1cm}
\setlength{\abovecaptionskip}{0.1cm}  
\setlength{\belowcaptionskip}{0.1cm}  
\includegraphics[width=0.98\textwidth]{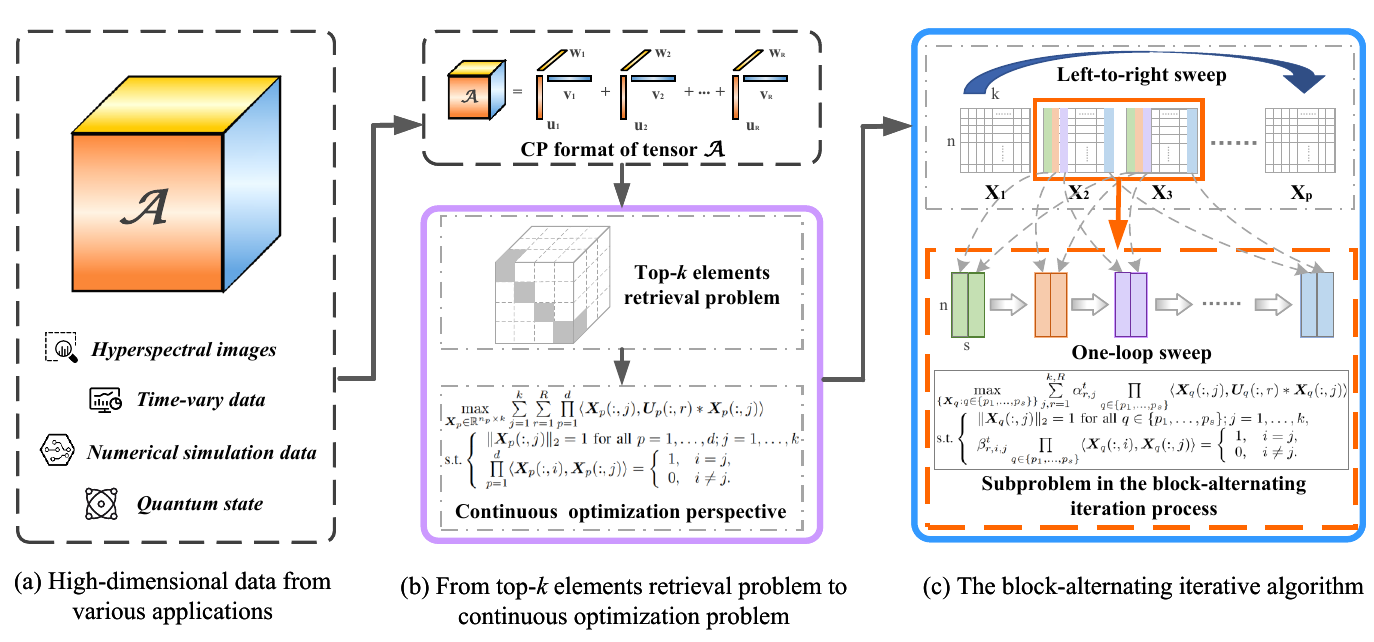}
\centering
\caption{The workflow of the block-alternating iterative algorithm for the top-$k$ elements retrieval problem in factorized tensors.}
\label{fig:framework}
\end{figure*}
%%%%%%%%%%%%%%%%%%%%%%%%%%%%%

Since the matrix $\bm{A}$ is diagonal, its eigenvector corresponding to the $i$th eigenvalue $\bm{A}(i,i)$ is $\bm{e}_i\in\mathbb{R}^{\prod\limits_{p=1}^dn_p}$, which is the $i$th column of $\prod\limits_{p=1}^dn_p\times\prod\limits_{p=1}^dn_p$ identity matrix. We note that the $d$th-order tensor corresponding to the eigenvector $\bm{e}_i$ can be decomposed into a rank-one form represented as $\bm{e}_{i_1}\circ\bm{e}_{i_2}\circ\cdots\circ\bm{e}_{i_d}$, where $i$ is the same as Eq.~\eqref{eq:diagonal}. $\bm{e}_{i_p}\in\mathbb{R}^{n_p}$ is the $i_p$th column of $n_p\times n_p$ identity matrix for all $p=1,2,\ldots,d$. In other words, if $\{\tensor{X}_j^*:j=1,\ldots,k\}$ are the solution of the optimization problem \eqref{prob:symeigen-t}, $\tensor{X}_j^*$ has a rank-one structure for all $j=1,\ldots,k$.
Based on this insight, the optimization problem~\eqref{prob:symeigen-t} can be further simplified via the CP representation of $\tensor{A}$, as detailed in Theorem \ref{thm:opt}.

\begin{thm}\label{thm:opt}
    {Let $\tensor{A}$ be a $d$th-order CP tensor with its CP representation given by Eq.~\eqref{eq:cp}. Subsequently, retrieving the $k$ largest elements of $\bm{\mathcal{A}}$ is equivalent to solving the following continuous constrained optimization problem}:
    %Let $\tensor{A}$ be a $d$th-order CP tensor, and its CP representation be described in \eqref{eq:cp}), then retrieving the $k$ largest elements of $\bm{\mathcal{A}}$ is equivalent to solving the following continuous constrained optimization problem
    \begin{equation}\label{prob:symeigen-r1}
    % \begin{split}
    \begin{array}{l}
          \max\limits_{\bm{X}_{p}\in\mathbb{R}^{n_p\times k}} \sum\limits_{j=1}^k\sum\limits_{r=1}^R \prod\limits_{p=1}^d\langle\bm{X}_{p}(:,j),\bm{U}_p(:,r)*\bm{X}_{p}(:,j)\rangle, \\
        \text{s.t.}\left\{\begin{array}{l}
             \|\bm{X}_{p}(:,j)\|_{2} = 1\quad \text{for all}\quad p=1,\ldots,d; j=1,\ldots,k,  \\
             \prod\limits_{p=1}^d\langle\bm{X}_{p}(:,i),\bm{X}_{p}(:,j)\rangle=\left\{\begin{array}{cc}
        1, & i=j, \\
        0, & i\neq j.
    \end{array}\right.
        \end{array}\right.
    \end{array}
    % \end{split}
\end{equation}
\end{thm}

\begin{proof}
    Since the symmetric matrix in the symmetric eigenvalue problem corresponding to the top-$k$ retrieval problem is a diagonal matrix, its $j$th eigenvector is a vector whose $j$th element is $1$ and other elements are $0$. Clearly, the $d$th-order tensor corresponding to the eigenvector has a rank-one structure. In other words, the solution to the optimization problem \eqref{prob:symeigen-t}, $\{\tensor{X}_j:j=1,\ldots,k\}$ satisfies that $\tensor{X}_j$ is a $d$th-order rank-one tensor for all $j=1,\ldots,k$. Therefore, the optimization problem \eqref{prob:symeigen-t} can be reformulated as 
    \begin{equation}\label{eq:reformulation}
    \begin{array}{l}
          \max\limits_{\tensor{X}_j\in\mathbb{R}^{n_1\times n_2\times\cdots\times n_d}}\sum\limits_{j=1}^k\langle\tensor{X}_j,\tensor{A}*\tensor{X}_j\rangle, \\
        \text{s.t.} \left\{\begin{array}{l}
        \tensor{X}_j=\bm{x}_{j,1}\circ\bm{x}_{j,2}\circ\cdots\circ\bm{x}_{j,d}\quad \text{for all}\quad j=1,\ldots,k, \\
              \prod\limits_{p=1}^d{\langle\bm{x}_{i,p},\bm{x}_{j,p}\rangle}=\left\{\begin{array}{cc}
            1, & i=j, \\
            0, & i\neq j.
        \end{array}\right.
        \end{array}\right. 
    \end{array}
    \end{equation}
    By the CP representation of $\tensor{A}$, i.e., 
    $\tensor{A} = \sum\limits_{r=1}^R\bm{U}_1(:,r)\circ\bm{U}_2(:,r)\circ\cdots\circ\bm{U}_d(:,r).$ Based on the rank-one structure of $\tensor{X}_j$, we can rewrite $\tensor{A}*\tensor{X}_j$ as 
    $\sum\limits_{r=1}^R\prod\limits_{p=1}^d\langle\bm{x}_{j,p},\bm{U}_p(:,r)*\bm{x}_{j,p}\rangle$ (i.e., the Hadamard product on CP tensors). Furthermore, the objective function of the optimization problem \eqref{eq:reformulation} is formulated as follows:
    \setlength\abovedisplayskip{0.2cm}
    \setlength\belowdisplayskip{0.2cm}
    \begin{equation}
        \sum\limits_{j=1}^k\sum\limits_{r=1}^R \prod\limits_{p=1}^d\langle\bm{x}_{j,p},\bm{U}_p(:,r)*\bm{x}_{j,p}\rangle.
    \end{equation}
    Thus, by denoting the matrix $[\bm{x}_{1,p},\ldots,\bm{x}_{k,p}]$ as $\bm{X}_p$, the optimization problem \eqref{eq:reformulation} is equivalent to problem \eqref{prob:symeigen-r1}. %$\hfill\blacksquare$ 
\end{proof} 

Theorem \ref{thm:opt} illustrates that retrieving the $k$ largest element of the CP tensor $\tensor{A}$ can be reformulated as a continuous constrained optimization problem. It is worth mentioning that the scale of this problem is $\sum\limits_{p=1}^dn_pk$, and the calculation of the objective function in the optimization problem \eqref{prob:symeigen-r1} only involves the Hadamard product of vectors of size $n_{p}$ for all $p=1,2,\ldots,d$. Therefore, it opens the door to developing algorithms for the top-$k$ element retrieval problem from a continuous optimization perspective.
% e.g., the gradient-based optimization algorithms presented in \cite{sidiropoulos2023minimizing} can also be used here.
%It is clear that the largest element and its location can be directly obtained by the optimal solution of \eqref{prob:symeigen-r1}. 
{Notably, the case of retrieving the $k$ smallest elements is addressed by replacing the maximization in the optimization problem \eqref{prob:symeigen-r1} with minimization.}
%By the way, if the $k$ smallest elements are required, we simply replace the maximization in optimization problem \eqref{prob:symeigen-r1}) with minimization.

\begin{rmk}\label{rmk:top-1}
    {For the case of $k=1$, the optimization problem~\eqref{prob:symeigen-r1} can degenerate into the following form:}
    \begin{equation}\label{prob:top1}
    % \begin{split}
    \begin{array}{l}
           \max\limits_{\bm{x}_{p}\in\mathbb{R}^{n_p}} \sum\limits_{r=1}^R \prod\limits_{p=1}^d\langle\bm{x}_{p},\bm{U}_p(:,r)*\bm{x}_{p}\rangle, \\
          \text{s.t.}\  \|\bm{x}_{p}\|_{2} = 1\quad \text{for all}\quad p=1,2,\ldots,d.
    \end{array}
    % \end{split}
\end{equation}
We remark that the continuous optimization model in \cite{sidiropoulos2023minimizing} can be derived directly from Eq.~\eqref{prob:top1} by rewriting its objective function as
$\sum\limits_{r=1}^R \prod\limits_{p=1}^d\sum\limits_{l=1}^{n_p}\bm{U}_p(l,r)\bm{y}_{p}(l)$, with $\bm{y}_p=\bm{x}_p*\bm{x}_p$ subject to $\sum\limits_{l=1}^{n_p}\bm{y}_p(l)=1$ for all $p=1,2,\ldots,d$.
\end{rmk}

\section{algorithm}\label{sec:dmrg}

In this section, we first introduce the block-alternating iterative algorithm, and then describe the design ideas of the heuristic method for the subproblems in the iterative process. The complexity analysis of the block-alternating iterative algorithm is discussed in Subsection~\ref{sec:cost}.

\subsection{The block-alternating iterative algorithm}

Due to its nonlinearity and constraints, the optimization problem \eqref{prob:symeigen-r1} is difficult to be solve for all variables $\{\bm{X}_p\in\mathbb{R}^{n_p\times k}:p=1,2,\ldots,d\}$ simultaneously. Inspired by the idea of nonlinear Gauss-Seidel iteration \cite{vrahatis2003linear}, we propose a block-alternating iterative algorithm to address it. Fig.~\ref{fig:framework} displays the workflow of the proposed algorithm.
In this algorithm, only a few target variables are updated at a time by splitting the original large-scale problem into a series of small-scale subproblems. More specifically, all target variables are fixed except the selected $s$ target variables $\left\{\bm{X}_{p_1},\ldots,\bm{X}_{p_s}\right\}$ that are updated in each iteration. Therefore, the key of the block-alternating iterative algorithm is to solve the corresponding small-scale subproblem at each iteration. Theorem~\ref{thm:subprob} presents the specific form of the subproblem.

\begin{thm}\label{thm:subprob}
    Let $\left\{\bm{X}_p^t:p=1,2,\ldots,d\right\}_{t\geq0}$ be the value of the $t$th step of the block-alternating iterative algorithm. Its corresponding subproblem in the iterative process is as follows 

\begin{equation}\label{prob:subpro}
\begin{small}
\begin{array}{l}
% \begin{split}
      \max\limits_{\{\bm{X}_{q}:q\in\{p_1,\ldots,p_s\}\}}\sum\limits_{j,r=1}^{k,R}\alpha_{r,j}^t \prod\limits_{q\in\{p_1,\ldots,p_s\}}\langle\bm{X}_{q}(:,j),  \bm{U}_{q}(:,r)*\bm{X}_{q}(:,j)\rangle, \\
    \text{s.t.}\left\{\begin{array}{l}
          \|\bm{X}_{q}(:,j)\|_{2} = 1\quad \text{for all}\quad q\in\{p_1,\ldots,p_s\}; j=1,\ldots,k, \\
           \beta_{r,i,j}^t\prod\limits_{q\in\{p_1,\ldots,p_s\}}\langle\bm{X}_{q}(:,i),\bm{X}_{q}(:,j)\rangle=\left\{\begin{array}{cc}
        1, & i=j, \\
        0, & i\neq j.
    \end{array}\right.
    \end{array}\right.
\end{array}
% \end{split}
\end{small}
\end{equation}where 
$$\alpha_{r,j}^t = \prod\limits_{q\notin\{p_1,\ldots,p_s\}}\langle\bm{X}^t_{q}(:,j),\bm{U}_{q}(:,r)*\bm{X}^t_{q}(:,j)\rangle,$$ 
and 
$$\beta_{r,i,j}^t=\prod\limits_{q\notin\{p_1,\ldots,p_s\}}\langle\bm{X}_{q}^t(:,i),\bm{X}_{q}^t(:,j)\rangle.$$ 

\begin{proof}
    When updating variables $\{\bm{X}_{q}:q\in\{p_1,\ldots,p_s\}\}$ in the $(t+1)$th iteration step, other variables are held fixed. Under these conditions, the two formulas
    
    \begin{equation*}
        \prod\limits_{p=1}^d\langle\bm{X}_p(:,j),\bm{U}_{p}(:,r)*\bm{X}_{p}(:,j)\rangle\quad \text{and}\quad \prod\limits_{p=1}^d\langle\bm{X}_p(:,j),\bm{X}_{p}(:,j)\rangle
    \end{equation*}
    can be decomposed as follows:
    % \begin{small}
    \begin{equation}
        \prod\limits_{q'\notin\{p_1,\ldots,p_s\}}\langle\bm{X}_{q'}(:,j),\bm{U}_{q'}(:,r)*\bm{X}_{q'}(:,j)\rangle\prod\limits_{q\in\{p_1,\ldots,p_s\}}\langle\bm{X}_{q}(:,j),\bm{U}_{q}(:,r)*\bm{X}_{q}(:,j)\rangle,
    \end{equation}
    % \end{small}
    and
    % \begin{small}
    \begin{equation}
        \prod\limits_{q'\notin\{p_1,\ldots,p_s\}}\langle\bm{X}_{q'}(:,j),\bm{X}_{q'}(:,j)\rangle\prod\limits_{q\in\{p_1,\ldots,p_s\}}\langle\bm{X}_{q}(:,j),\bm{X}_{q}(:,j)\rangle.
    \end{equation}
    % \end{small}
    % \begin{scriptsize}
    % \end{scriptsize}
    By denoting $\prod\limits_{q'\notin\{p_1,\ldots,p_s\}}\langle\bm{X}^t_{q'}(:,j),\bm{U}_{q'}(:,r)*\bm{X}^t_{q'}(:,j)\rangle$ and $\prod\limits_{q'\notin\{p_1,\ldots,p_s\}}\langle\bm{X}_{q'}^t(:,i),\bm{X}_{q'}^t(:,j)\rangle$ as $\alpha_{r,j}^t$ and $\beta_{r,i,j}^t$, we can obtain the form of the subproblem \eqref{prob:subpro}. 
    %$\hfill\blacksquare$ 
\end{proof}

%%%%%%%%%%%%%%%%%%%%%%%%%%%%%%%%
%!htp
\begin{algorithm}[t]
\caption{The block-alternating iterative algorithm.}\label{algo:dmrg}
\begin{algorithmic}[1]
\REQUIRE Tensor given in the CP format $\bm{\mathcal{A}} = \sum\limits_{r=1}^R \bm{U}_1({:,r})\circ\bm{U}_2({:,r})\circ\cdots\circ \bm{U}_d({:,r})$, $k\geq1$, and initial values $\left\{\bm{X}_p\in\mathbb{R}^{n_p\times k}:p=1,2\ldots,d\right\}$ such that $\prod\limits_{p=1}^d\langle\bm{X}_p(:,i),\bm{X}_p(:,j)\rangle=\left\{\begin{array}{cc}
    1, & i=j, \\
    0, & i\neq j.
\end{array}\right.$
\ENSURE The $k$ largest elements of $\tensor{A}$ and their corresponding locations: $\bm{val}\in\mathbb{R}^k$ and $\bm{ind}\in\mathbb{Z}_+^{d\times k}$.
\STATE $\bm{val},\bm{ind}\leftarrow\left[\ \right]$
\WHILE{not convergent}
\STATE Pick $s$ variables $\left\{\bm{X}_{p_1},\ldots,\bm{X}_{p_s}\right\}$ to be updated %(i.e., ordered or random selection)
\STATE Compute these $k$ intermediate $s$th-order CP tensors
$$\left\{\sum\limits_{r=1}^R\alpha_{r,j}^t\bm{U}_{p_1}(:,r)\circ\cdots\circ\bm{U}_{p_s}(:,r):j=1,\ldots,k\right\}$$
\FOR{$j=1$ to $k$}
\STATE $\bm{val}_{j},\bm{ind}_{j}\ \leftarrow$ the largest element of the CP tensor $\sum\limits_{r=1}^R\alpha_{r,j}^t\bm{U}_{p_1}(:,r)\circ\cdots\circ\bm{U}_{p_s}(:,r)$ and its location
\FOR{$i=1$ to $j-1$}
\STATE $\beta_{r,i,j}^t\ \leftarrow\ \prod\limits_{q\notin\{p_1,\ldots,p_s\}}\langle\bm{X}_{q}^t(:,i),\bm{X}_{q}^t(:,j)\rangle$
\IF{$\beta_{r,i,j}^t \neq 0$ and $\bm{ind}_i=\bm{ind}_j$}
\STATE $\bm{val}_{j},\bm{ind}_{j}\ \leftarrow$ the largest element of $\sum\limits_{r=1}^R\alpha_{r,j}^t\bm{U}_{p_1}(:,r)\circ\cdots\circ\bm{U}_{p_s}(:,r)$ that is not greater than $\bm{val}_j$ and its location
\ELSE
\STATE Continue
\ENDIF
\ENDFOR
\ENDFOR
\STATE $\bm{val}\leftarrow\bm{val}\cup\bm{val}_j,\ \bm{ind}\leftarrow\bm{ind}\cup\bm{ind}_j$
% \STATE $\bm{val},\bm{ind}\leftarrow$ Using Algorithm~\ref{algo:heuristic} to solve the subproblem \eqref{prob:subpro})
\ENDWHILE
\end{algorithmic}
\end{algorithm}

%%%%%%%%%%%%%%%%%%%%%%%%%%%%%%%%%%%%

Compared to the original optimization problem \eqref{prob:symeigen-r1}, the size of the subproblem \eqref{prob:subpro} is reduced to $\sum\limits_{q\in\{p_1,\ldots,p_s\}}n_qk$, where the block parameter $s$ is preset by users. However, the constraints make this subproblem difficult to solve exactly. This is due to that determining constraint validity based on $\beta_{r,i,j}^t$ causes the time cost to grow exponentially with $k$, making this approach impractical for medium or large values of $k$. To overcome this, we propose a heuristic method for the subproblem \eqref{prob:subpro}.

%%%%%%%%%%%%%%%%%%%%%%%%%%%%%%%%%%%%%%%%%%%%%%%

\begin{figure*}[t]
\centering
\subfigure[]{
\includegraphics[width=0.45\linewidth]{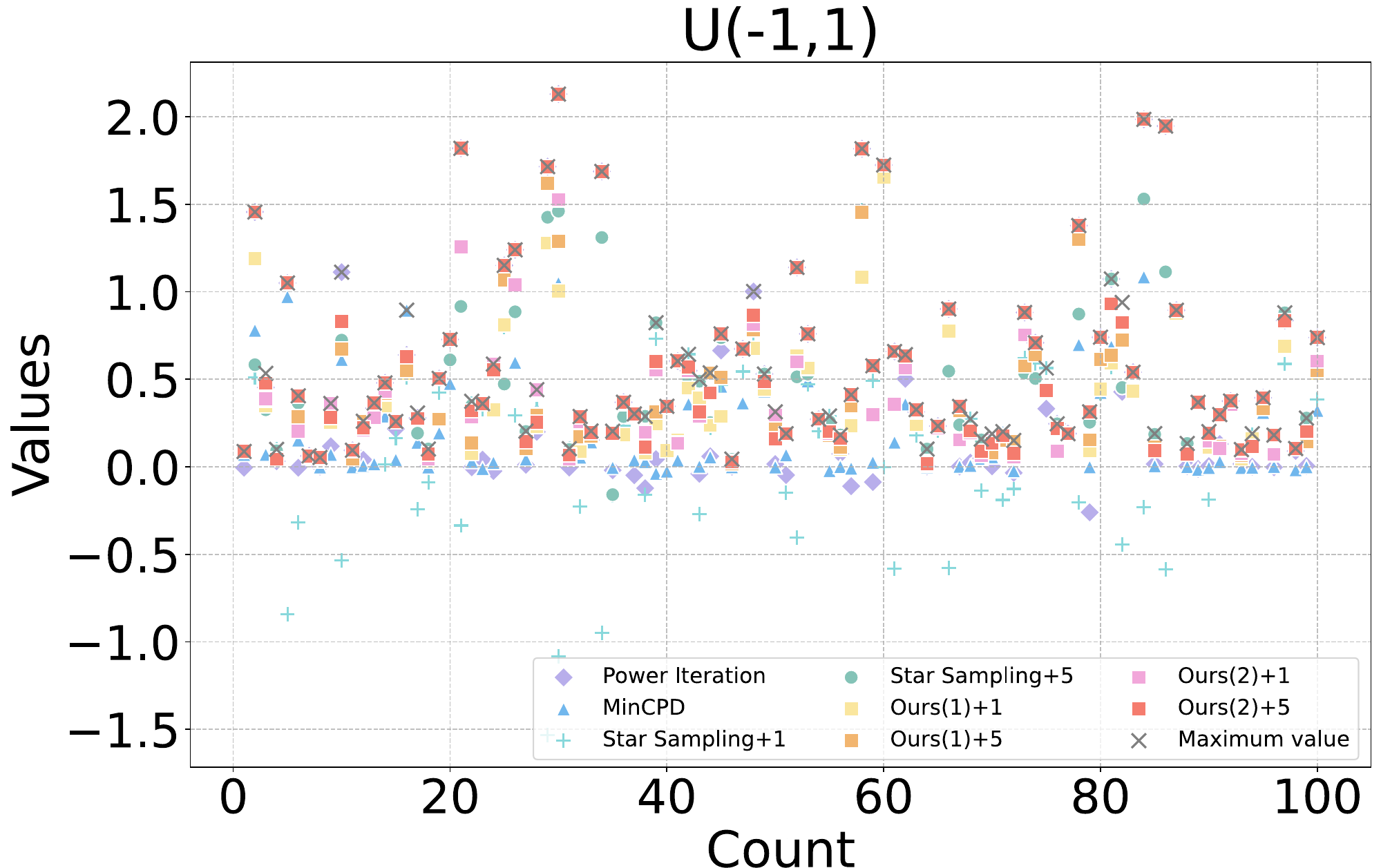}}
\hspace{0.05\linewidth}
\subfigure[]{
\includegraphics[width=0.45\linewidth]{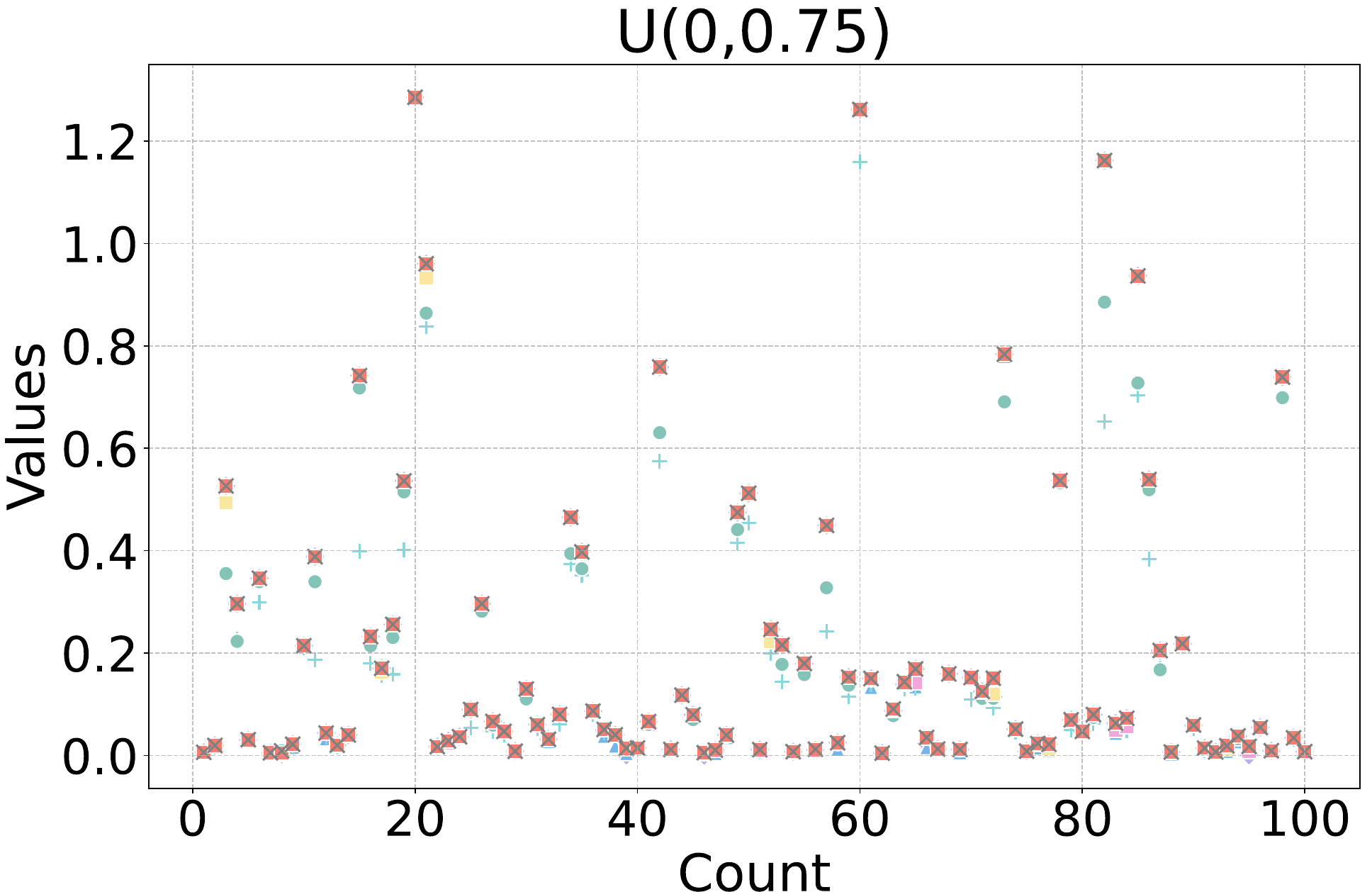}}
\vfill
\subfigure[]{
\includegraphics[width=0.45\linewidth]{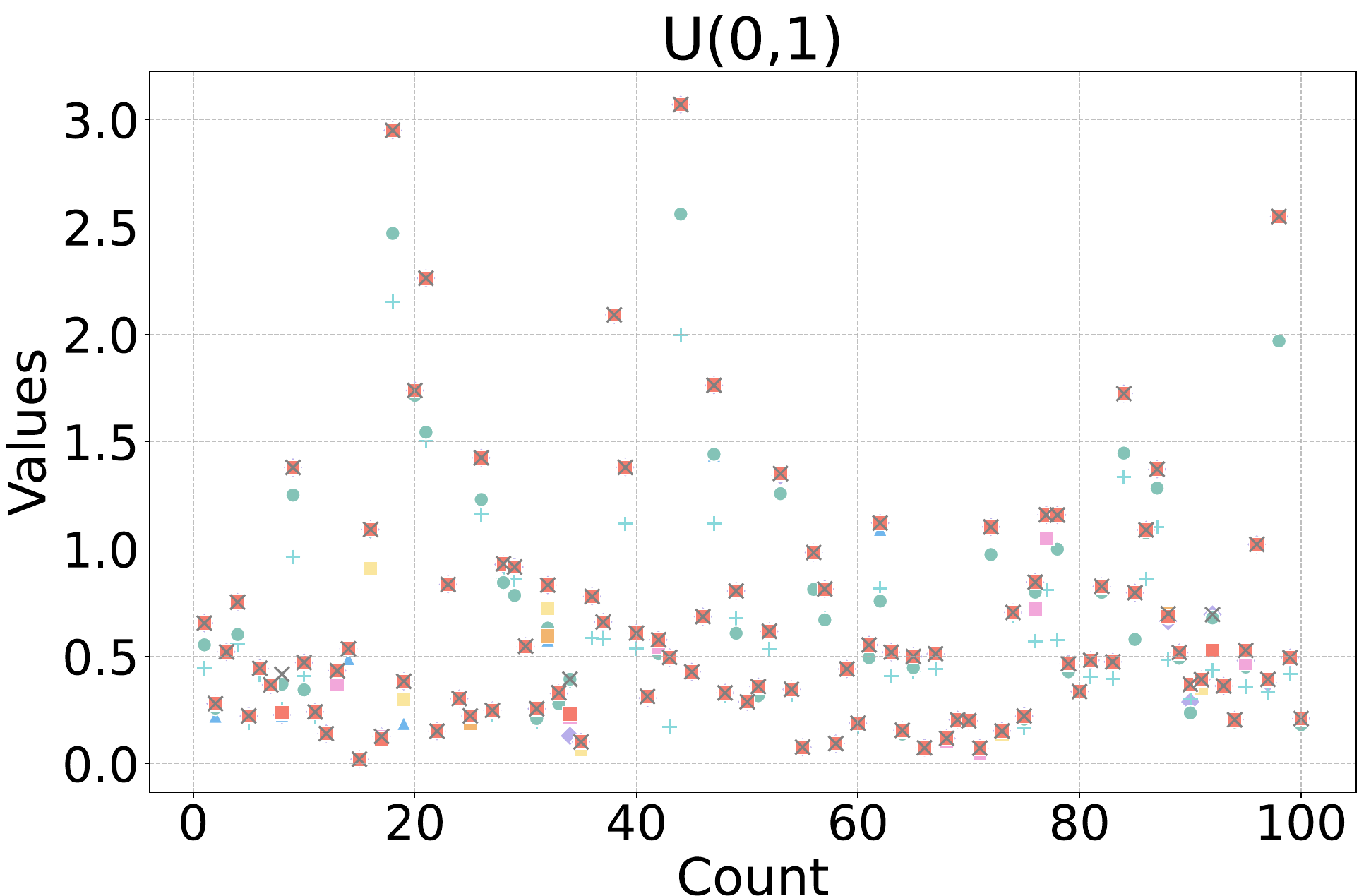}}
\hspace{0.05\linewidth}
\subfigure[]{
\includegraphics[width=0.45\linewidth]{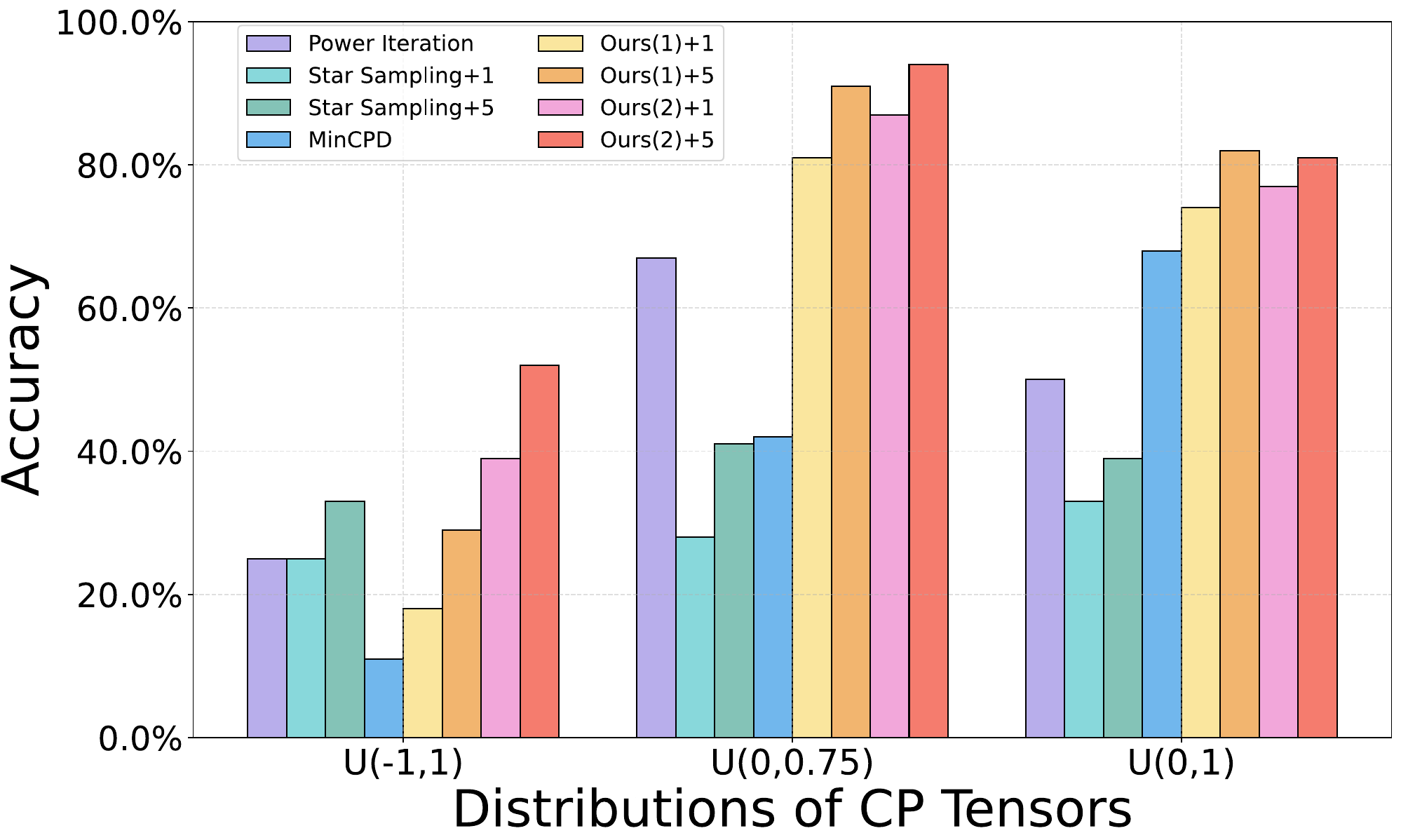}}
\caption{The largest value and accuracy obtained by baselines and our algorithm for random CP tensors with different uniform distributions.}
\label{fig1}
\end{figure*}

%%%%%%%%%%%%%%%%%%%%%%%%%%%%%%%%%%%%%%%%%%%%%%%

% \subsubsection{The Heuristic Method}
The heuristic method is inspired by the structure of the objective function of the subproblem \eqref{prob:subpro}, which is a separable summation representation with respect to columns of $\{\bm{X}_{p_1},\ldots,\bm{X}_{p_s}\}$, i.e.,
% \begin{tiny}
\begin{equation}
    f_1(\bm{X}_{p_1}(:,1),\ldots,\bm{X}_{p_s}(:,1))+\cdots+f_k(\bm{X}_{p_1}(:,k),\ldots,\bm{X}_{p_s}(:,k)),
\end{equation}
% \end{tiny}
where 
\begin{equation}
    f_j(\bm{X}_{p_1}(:,j),\ldots,\bm{X}_{p_s}(:,j)) = \sum\limits_{r=1}^R\alpha_{r,j}^t\prod\limits_{q\in\{p_1,\ldots,p_s\}}\langle\bm{X}_{q}(:,j),\bm{U}_{q}(:,r)*\bm{X}_{q}(:,j)\rangle.
\end{equation}
With the separable summation representation, we can identify $\{(\bm{X}_{p_1}(:,j),\ldots,\bm{X}_{p_s}(:,j)):j=1,\ldots,k\}$ in a certain order such that they meet local optimality. Specifically, taking the order $1\rightarrow2\rightarrow\cdots\rightarrow k$ as an example, $\{(\bm{X}_{p_1}(:,j),\ldots,\bm{X}_{p_s}(:,j)):j=1,\ldots,k\}$ can be determined as follows. For $j=1$,
\begin{equation}
    % \begin{split}
        (\bm{X}_{p_1}^*(:,1),\ldots,\bm{X}_{p_s}^*(:,1))=\arg\max\limits_{\bm{X}_{p_1}(:,1),\ldots,\bm{X}_{p_s}(:,1)} f_1(\bm{X}_{p_1}(:,1),\ldots,\bm{X}_{p_s}(:,1)),
        % (\bm{X}_{p_1}^*(:,j),\ldots,\bm{X}_{p_s}^*(:,j))&=\arg\max\limits_{\bm{X}_{p_1}(:,j),\ldots,\bm{X}_{p_s}(:,j)} f_j(\bm{X}_{p_1}(:,j),\ldots,\bm{X}_{p_s}(:,j))\quad\text{s.t.}\quad\beta_{r,i,j}\prod\limits_{q\in\{p_1,\ldots,p_s\}}\langle\bm{X}_{q}(:,i),\bm{X}_q(:,j)\rangle=0\quad\text{for all}\quad i<j.
    % \end{split}
\end{equation}
which is equivalent to retrieving the largest element and its location of the $s$th-order CP tensor $\sum\limits_{r=1}^R\alpha_{r,j}^t\bm{U}_{p_1}(:,r)\circ\cdots\circ\bm{U}_{p_s}(:,r)$.
For $j>1$,
\begin{equation}
(\bm{X}_{p_1}^*(:,j),\ldots,\bm{X}_{p_s}^*(:,j)) =\\ \arg\max\limits_{\bm{X}_{p_1}(:,j),\ldots,\bm{X}_{p_s}(:,j)} f_j(\bm{X}_{p_1}(:,j),\ldots,\bm{X}_{p_s}(:,j))
\end{equation}
% \begin{tiny}
\begin{equation}
    \text{s.t.}\ \beta_{r,i,j}^t\prod\limits_{q\in\{p_1,\ldots,p_s\}}\langle\bm{X}_{q}(:,i),\bm{X}_q(:,j)\rangle=0\ \text{for all}\ i<j,
\end{equation}
% \end{tiny}
which includes two cases: (1) if $\beta_{r,i,j}^t=0$, indicating an invalid constraint, the problem is equivalent to retrieving the largest element and its location of the $s$th-order CP tensor $\sum\limits_{r=1}^R\alpha_{r,j}^t\bm{U}_{p_1}(:,r)\circ\cdots\circ\bm{U}_{p_s}(:,r)$; (2) otherwise, we need to find the largest element and its location that meets the condition
\begin{equation}
    \prod\limits_{q\in\{p_1,\ldots,p_s\}}\langle\bm{X}_{q}(:,i),\bm{X}_q(:,j)\rangle=0\quad\text{for all}\quad i<j.
\end{equation} 
Combined with the heuristic method, we summarize the overall procedure of the proposed block-alternating iterative algorithm in Algorithm \ref{algo:dmrg}. 

\begin{rmk}\label{rmk:blockpara}
    The accuracy and efficiency of the block-alternating iterative algorithm are sensitive to the block parameter $s$. A larger $s$ allows the block-alternating algorithm to explore a wider solution space, potentially improving accuracy at the cost of increased computational and memory overhead.
    In practice, setting $s$ to $2$ or $3$ is usually sufficient to achieve acceptable accuracy. A more adaptive strategy for setting $s$ is based on the dimensions of the input CP tensor, which ensures a consistent subproblem scale during the iterative process.
\end{rmk}

\begin{rmk}\label{rmk:subproblem}
    The proposed block-alternating iterative algorithm can be enhanced by integrating with the existing methods to retrieve the largest/smallest element from the $s$th-order CP tensor in the subproblem (line 10), which may improve its efficiency. 
\end{rmk}

\begin{rmk}\label{rmk:function}
    It is worth mentioning that the proposed block-alternating iterative algorithm can not only retrieve the $k$ largest elements and its locations, but also obtain the $k$ largest elements of the real (imaginary) part of complex CP tensors. Specifically, it can be implemented by changing the largest element in line 10 of Algorithm~\ref{algo:dmrg} to the largest element of the real (imaginary) part.
\end{rmk}

%%%%%%%%%%%%%%%%%%%%%%%%%%%%%%%%
\begin{figure*}[t]
\centering
\subfigure[]{
\includegraphics[width=0.45\linewidth]{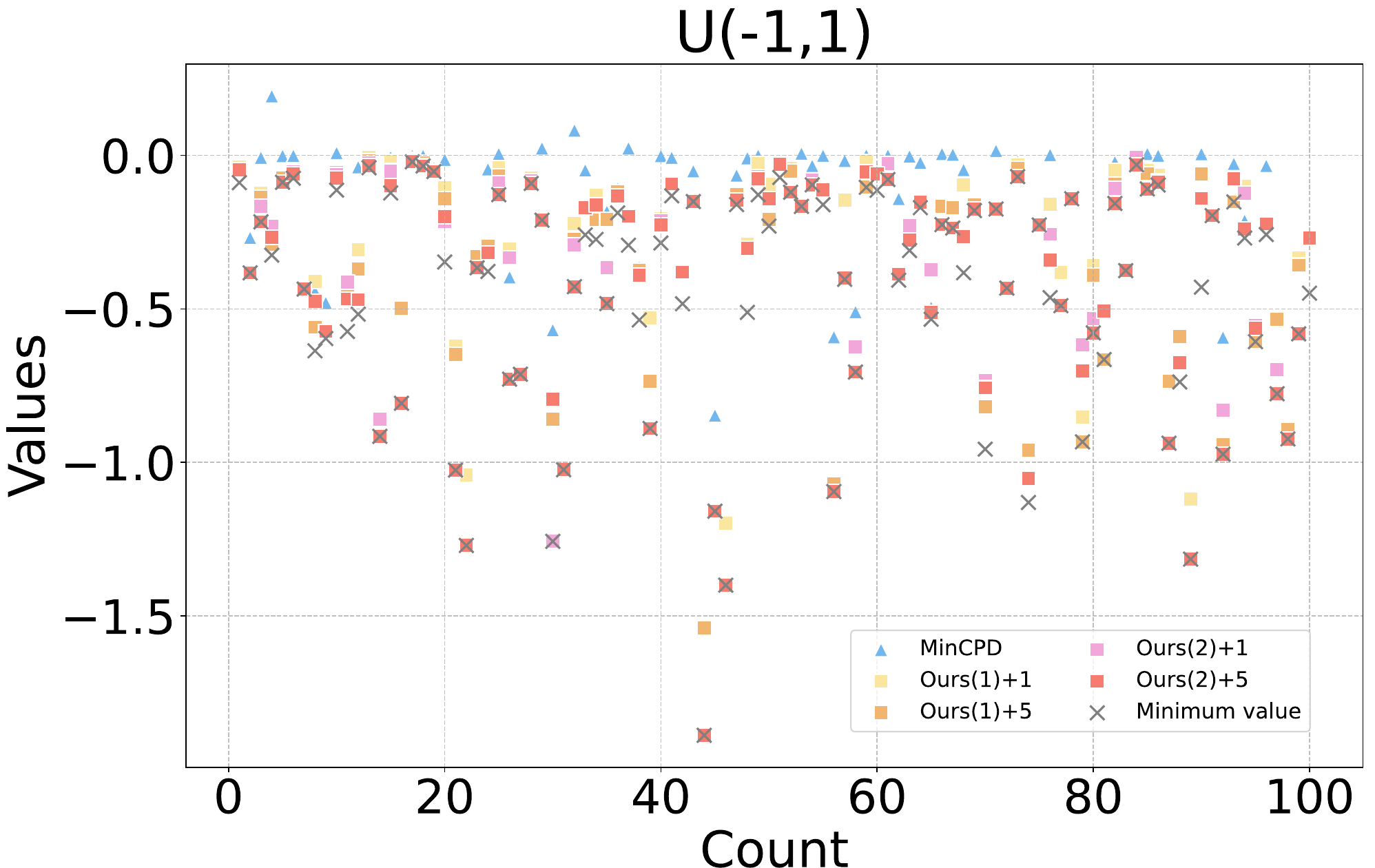}}
\hspace{0.05\linewidth}
\subfigure[]{
\includegraphics[width=0.45\linewidth]{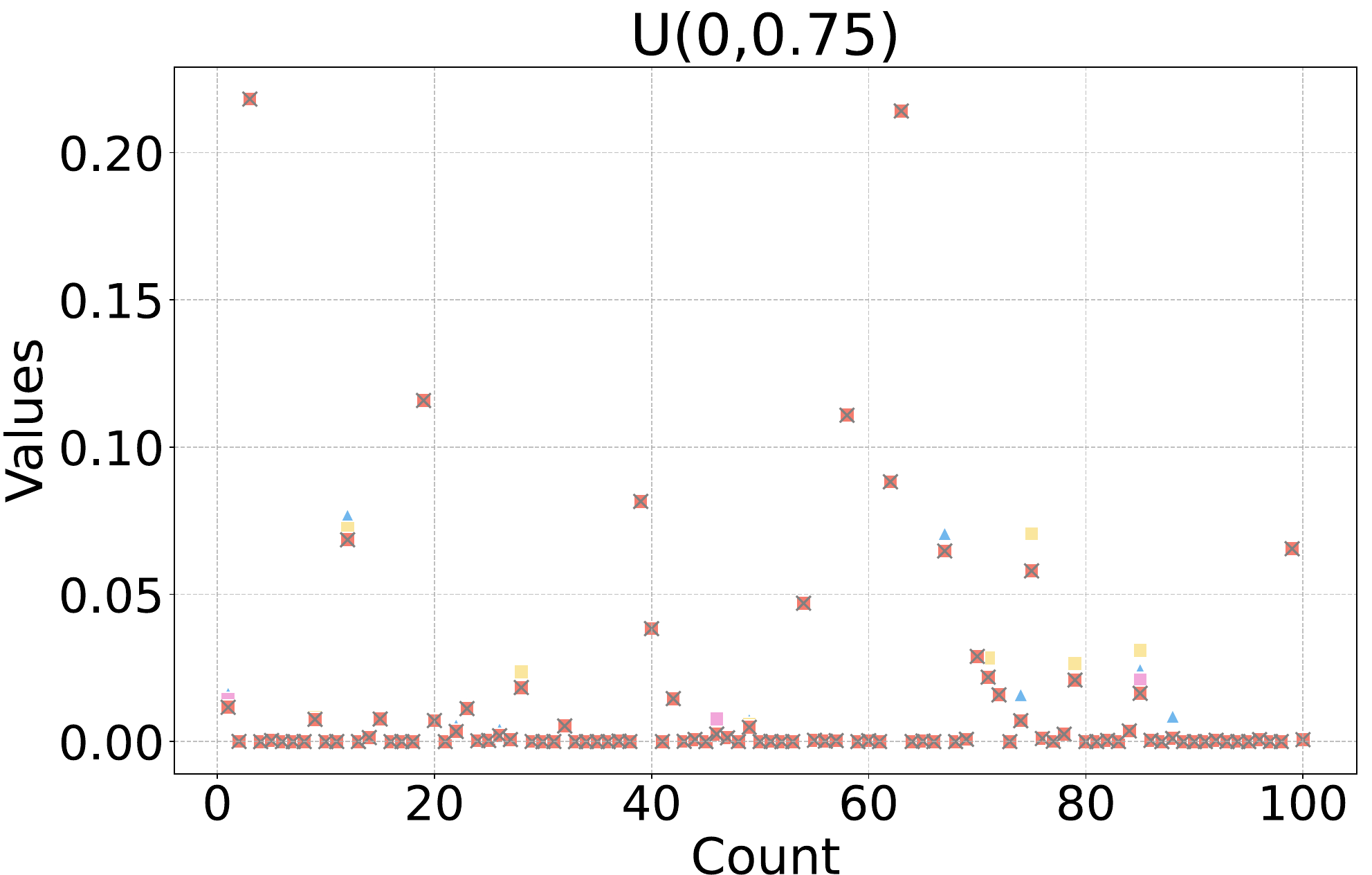}}
\vfill
\subfigure[]{
\includegraphics[width=0.45\linewidth]{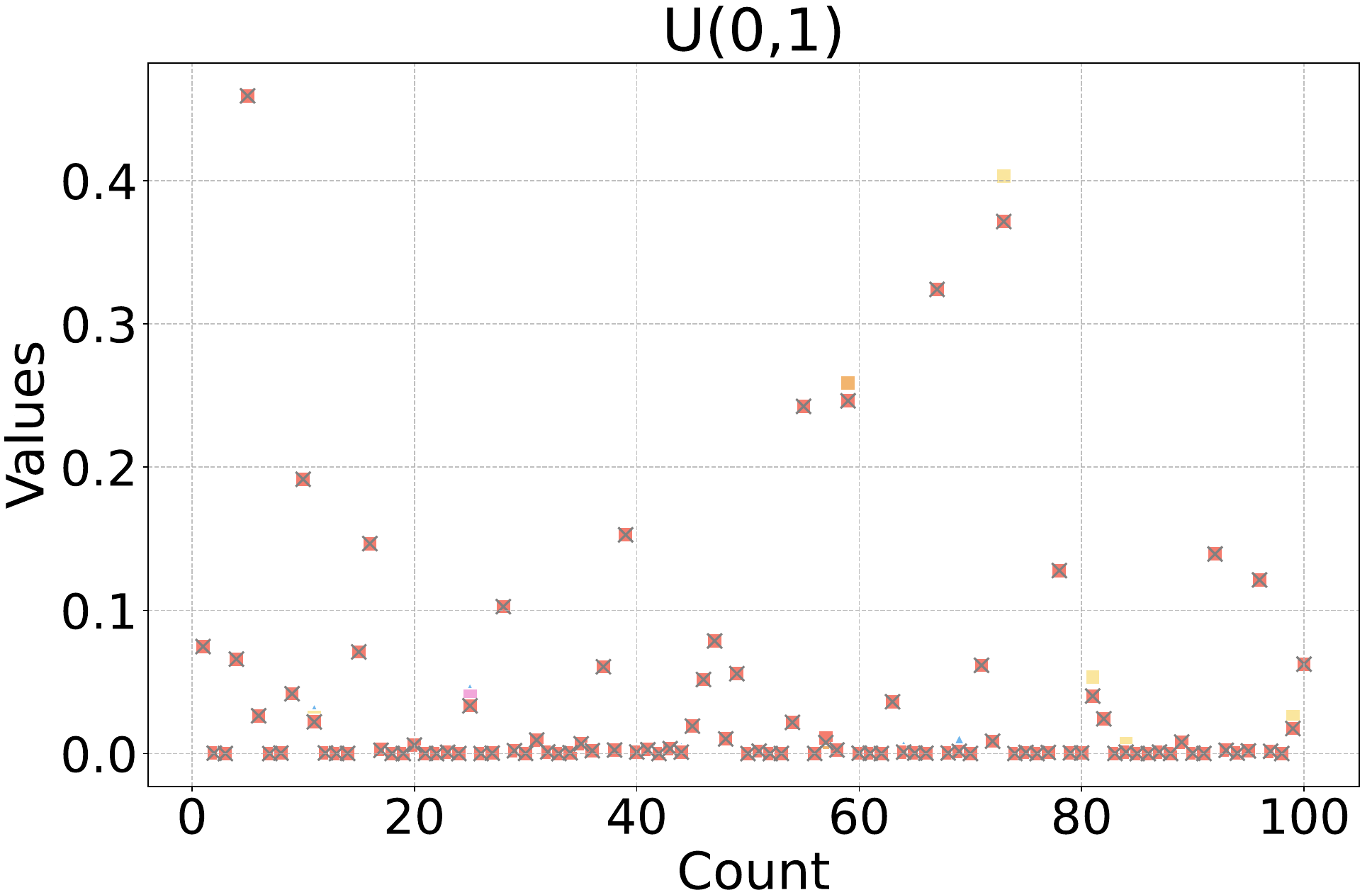}}
\hspace{0.05\linewidth}
\subfigure[]{
\includegraphics[width=0.45\linewidth]{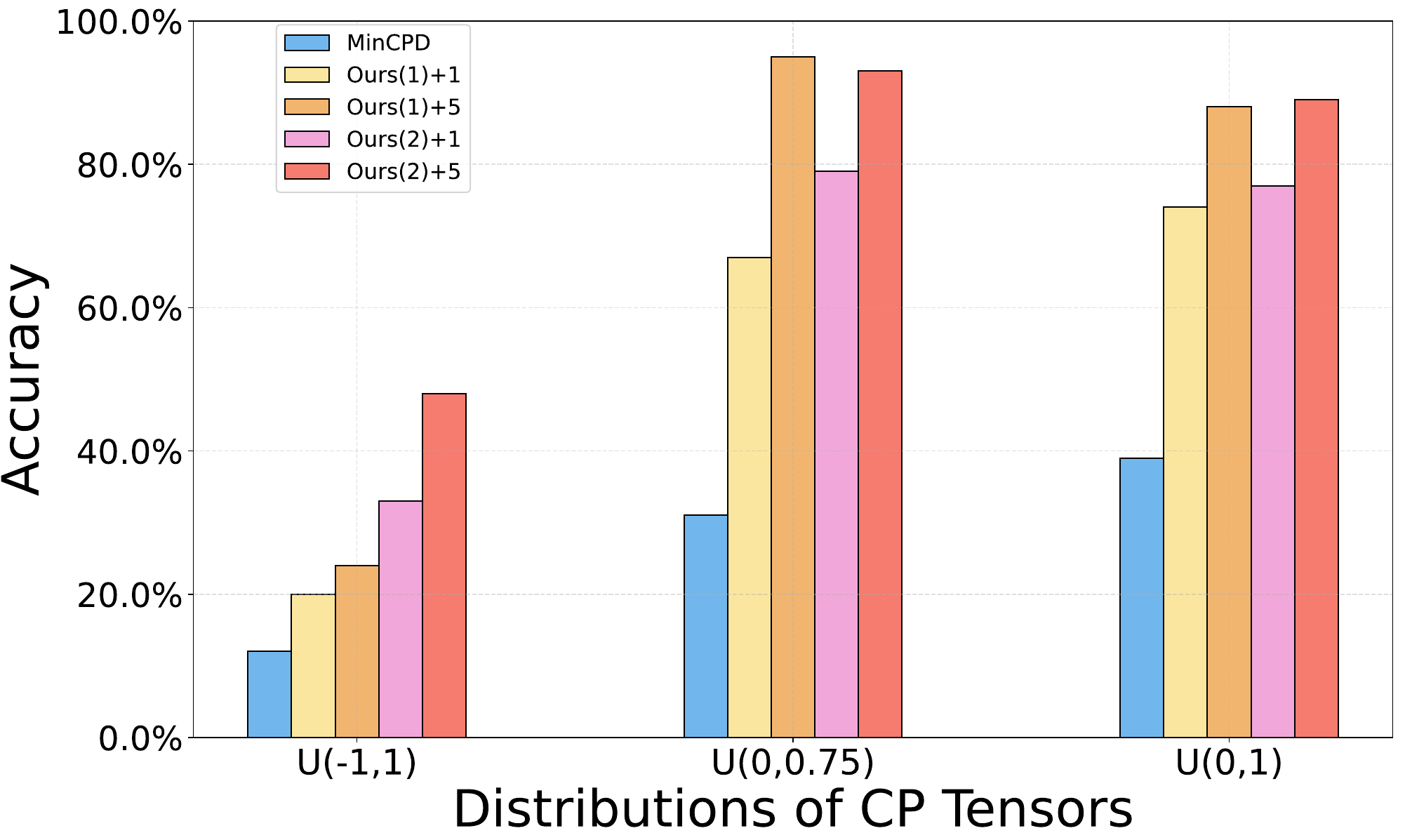}}
\caption{The smallest value and accuracy obtained by baselines and our algorithm for random CP tensors with different uniform distributions.}
\label{fig2}
\end{figure*}

%%%%%%%%%%%%%%%%%%%%%%%%%%%%%%%%%%%%%%%%%%%%

\subsection{Complexity analysis}\label{sec:cost}

We next analyze the time and memory costs of the block-alternating iterative algorithm, which mainly includes two parts: (1) computing the $k$ intermediate $s$th-order CP tensors (line 4); and (2) solving the subproblem \eqref{prob:subpro} by the proposed heuristic method (lines 5-15). The calculation in the first part is to compute the coefficients $\{\alpha_{r,j}^t:r=1,\ldots,R\}$, which involves the Hadamard and inner product operations on CP tensors. Thus, it requires a time cost of $\sum\limits_{q\notin\{p_1,\ldots,p_s\}}n_q$ but no additional memory cost. In the second part, we need to find the largest element and its location of the $s$th-order CP tensor $\sum\limits_{r=1}^R\alpha_{r,j}^t\bm{U}_{p_1}(:,r)\circ\cdots\circ\bm{U}_{p_s}(:,r)$ accurately, and determine whether it satisfies the constraints by $\beta_{r,i,j}^t$. The direct method is to restore the CP tensor to a full tensor and then apply classical sorting algorithms \cite{cormen2022introduction} to implement it. The required time and memory costs are $\prod\limits_{q\in\{p_1,\ldots,p_s\}}n_qR+\prod\limits_{q\in\{p_1,\ldots,p_s\}}n_q\log\left(\prod\limits_{q\in\{p_1,\ldots,p_s\}}n_q\right)$ and $\prod\limits_{q\in\{p_1,\ldots,p_s\}}n_q$, respectively. Additionally, the calculation of $\beta_{r,i,j}^t$ only involves the inner product of vectors. Thus, its time cost can be ignored.
\end{thm}

\section{Numerical experiments}\label{sec:numexp}

To demonstrate the effectiveness of the proposed block-alternating iterative algorithm, we conduct a series of numerical experiments using synthetic and real-world CP tensors in this section. The proposed algorithm is compared with several baselines, including power iteration \cite{espig2020iterative}, star sampling \cite{lu2017sampling}, and MinCPD via Frank-Wolfe \cite{sidiropoulos2023minimizing}. %with the following parameter settings.
During the iterative process of power iteration, if the CP rank of the tensor corresponding to the eigenvector exceeds 10, a recompression operation is introduced to suppress rank growth. To ensure the power iteration finds the largest element (not just the largest in magnitude), we apply a shift transformation to make the input tensor non-negative. The transformation is $\tensor{A}+s\tensor{E}$, where $s > 0$ is a sufficiently large scalar (e.g., $\|\tensor{A}\|_F$) and $\tensor{E}$ is a tensor whose entries are all one. Moreover, we need to determine the location of the largest element by computing the best rank-one approximation of the tensor corresponding to the eigenvector in the power iteration. In the star sampling method, the number of nodes and samples are set to $2$ and $\min(10^4, \lfloor20\%\times\#P(\tensor{A})\rfloor)$, following the guidelines in \cite{lu2017sampling}, where $\#P(\tensor{A})$ denotes the total number of elements in $\tensor{A}$. Since the star sampling and block-alternating iterative algorithms are suitable for the case of $k>1$, a simple trick to improve their accuracy involves expanding the search space by increasing $k$ to $k+K$. For convenience, we denote these two methods as $\text{XX}+K$ and evaluate them with $K=1$ and $K=5$.
% Moreover, 
All experiments in this paper are conducted on a laptop. The tensor operations for the tested algorithms are implemented by the \textit{TensorLy} package \cite{kossaifi2019tensorly} with \textit{NumPy} as the backend.

%%%%%%%%%%%%%%%%%%%%%%%%%%%%%%%%
\begin{figure*}
\centering
\subfigure[]{
\includegraphics[width=0.45\linewidth]{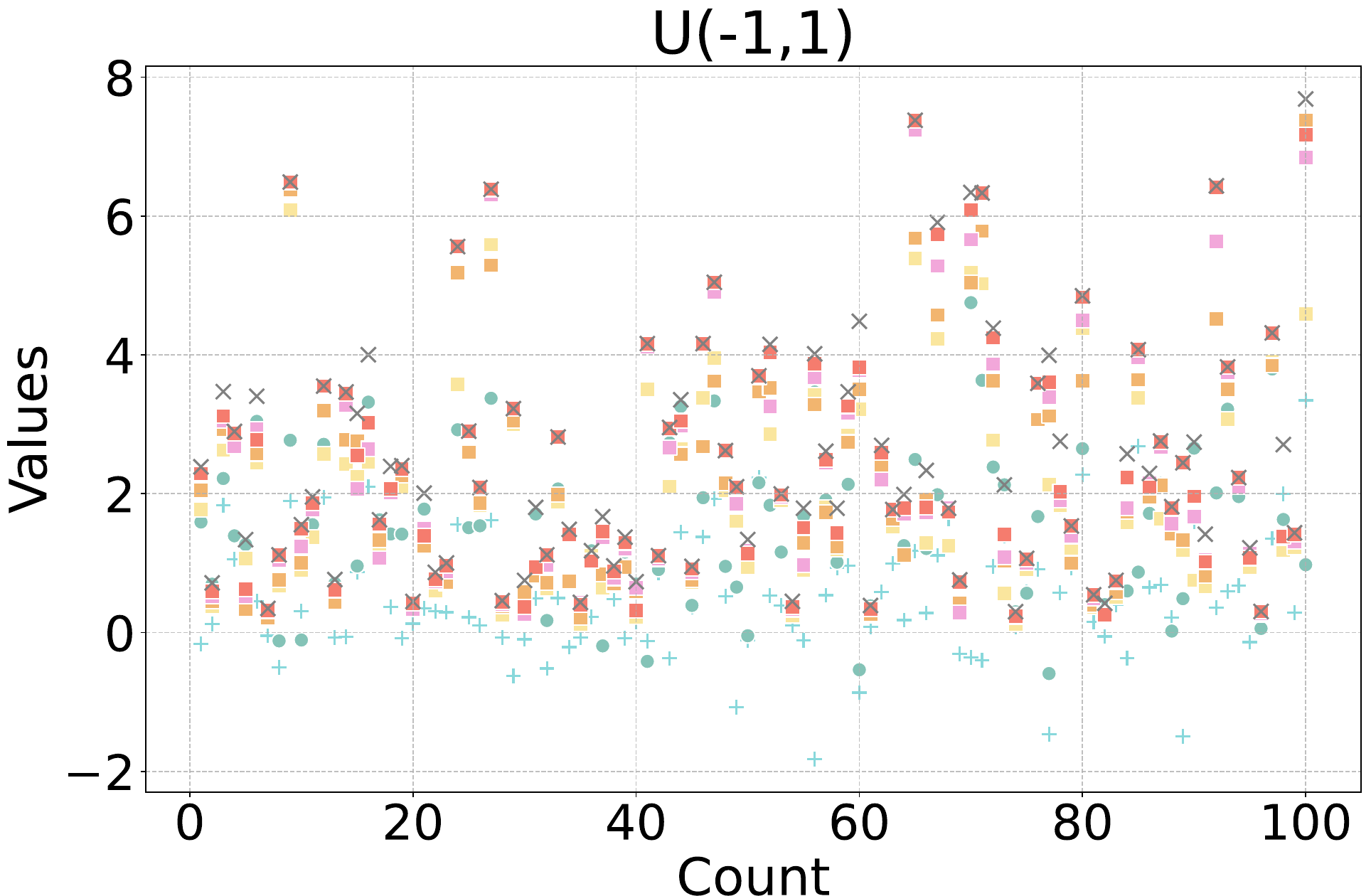}}
\hspace{0.05\linewidth}
\subfigure[]{
\includegraphics[width=0.45\linewidth]{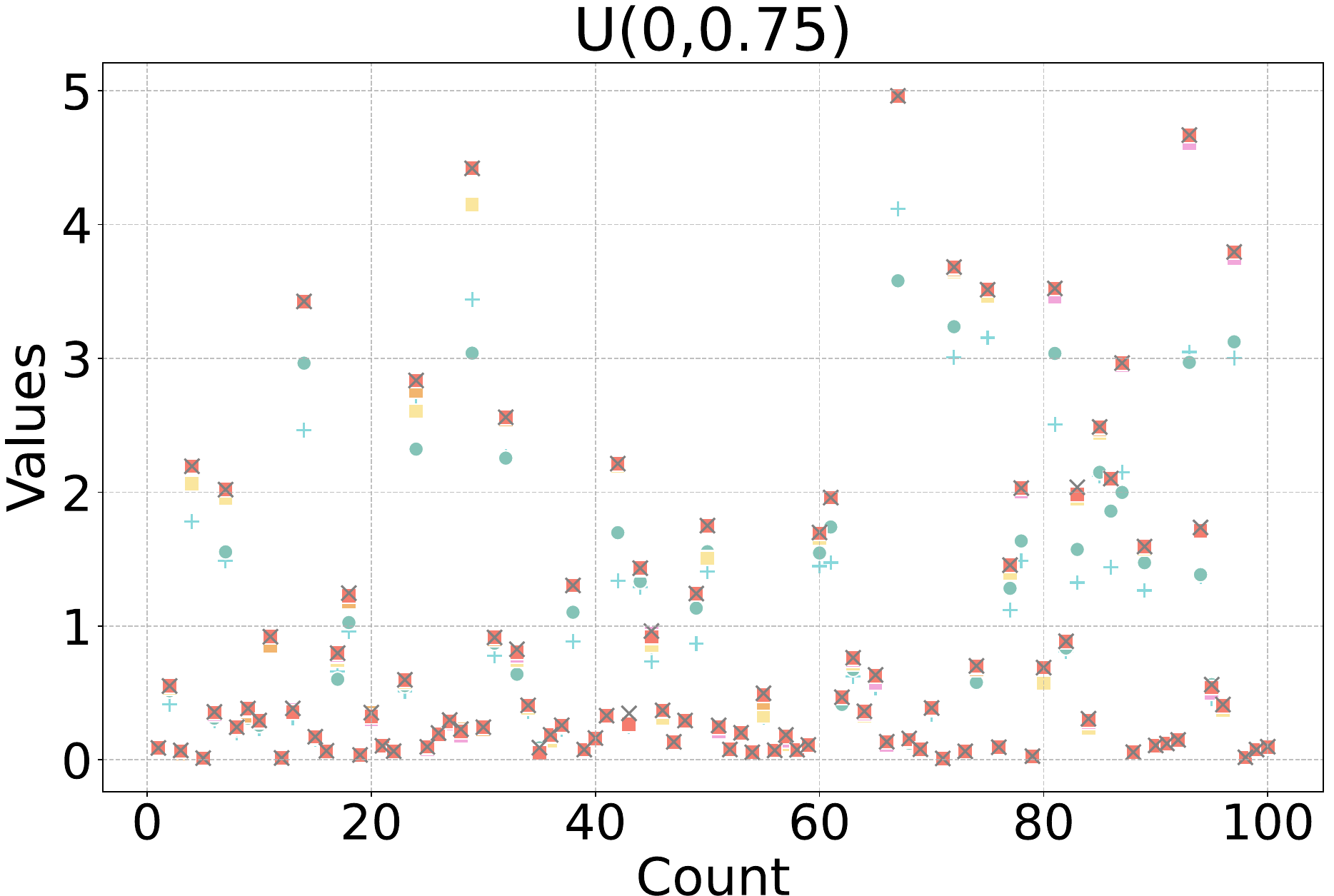}}
\vfill
\subfigure[]{
\includegraphics[width=0.45\linewidth]{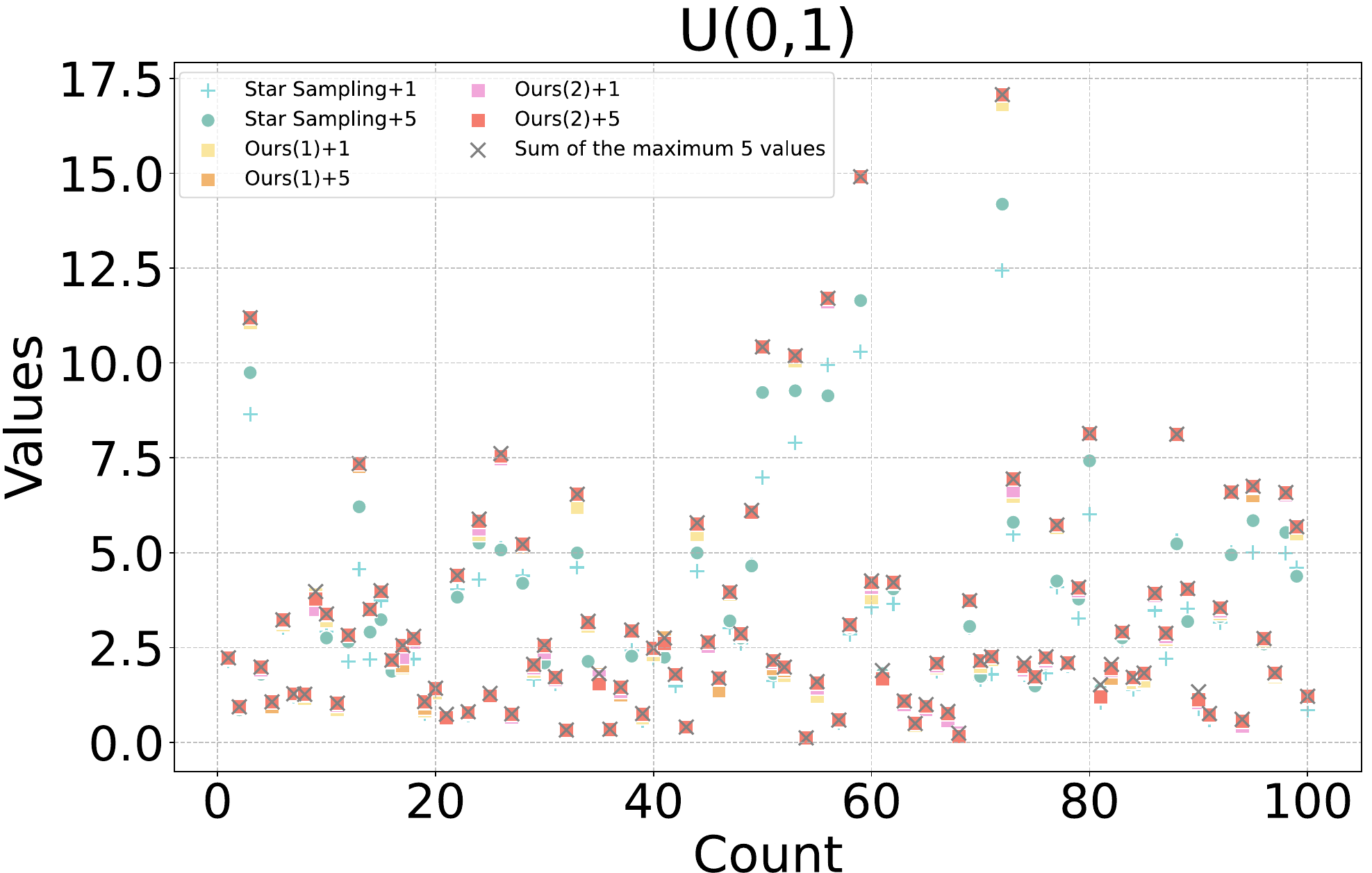}}
\hspace{0.05\linewidth}
\subfigure[]{
\includegraphics[width=0.45\linewidth]{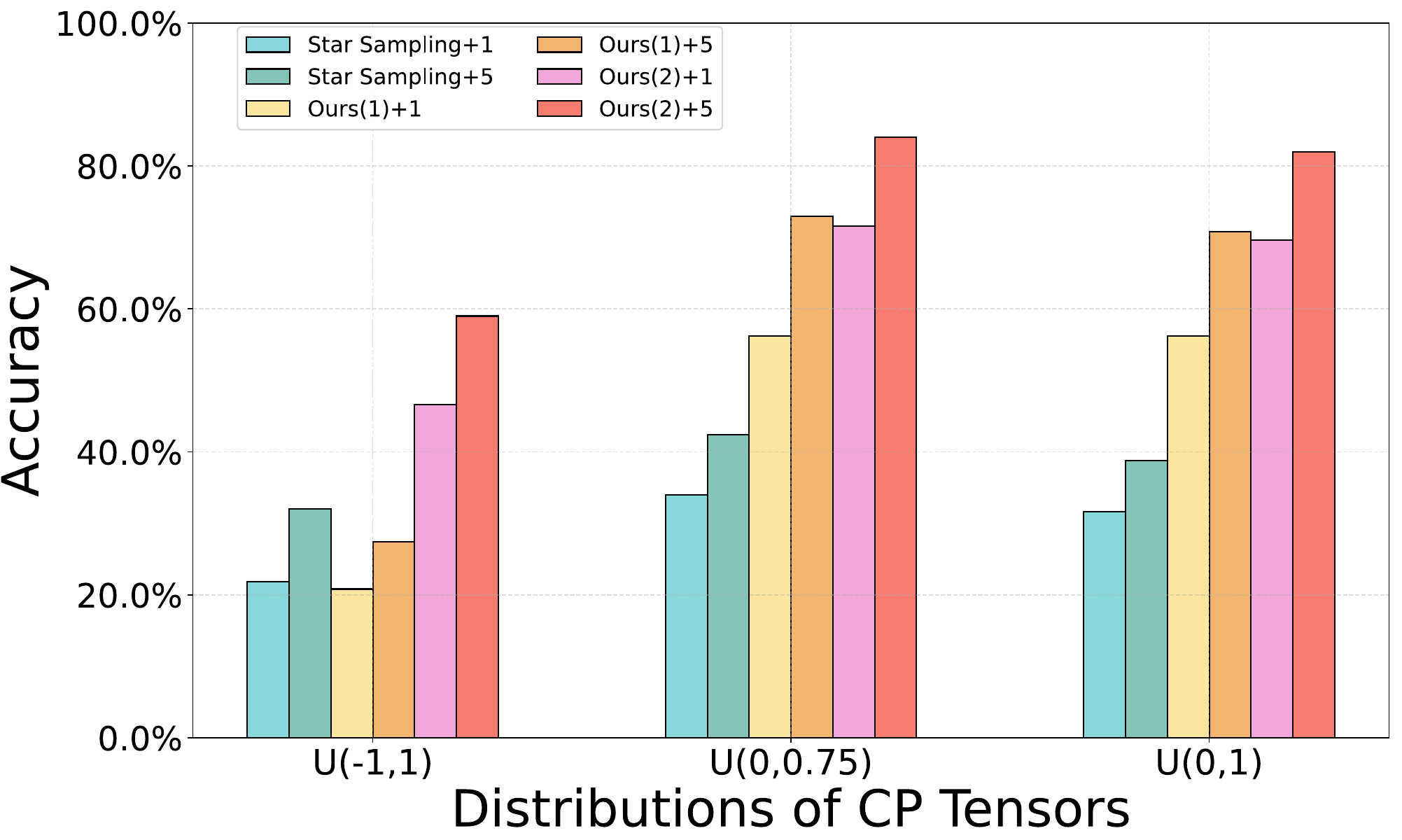}}
\caption{The sum of the 5 largest values and accuracy obtained by baselines and our algorithm for random CP tensors with different uniform distributions.}
\label{fig3}
\end{figure*}

%%%%%%%%%%%%%%%%%%%%%%%%%%%%%%%%%%%%%%%%%%%%

\subsection{Tests on random CP tensors}

In the first example, we compare the proposed block-alternating iterative algorithm with baselines on randomly generated CP tensors.
The accuracy is defined as $\frac{\#\text{hit}}{S}$, where $\#\text{hit}$ is the number of successful identifications of the largest/smallest elements, and $S=100$ is the number of random CP tensors.
Specifically, the input CP tensor is generated as follows: The CP factors $\left\{\bm{U}_{p}\in\mathbb{R}^{n_p\times R}:p=1,2,\ldots,d\right\}$ of each tensor are randomly generated following uniform distributions, including $U(-1,1)$, $U(0,0.75)$, and $U(0,1)$. Moreover, the order $d$, dimensions $\left\{n_p: p=1,2,\ldots,d\right\}$, and CP rank $R$ are randomly chosen integers between  $\left[3,10\right]$, $\left[2,15-d\right]$ and $\left[2,10\right]$, respectively.

% \begin{itemize}
%     \item Top-1 largest and smallest elements.
%     \item Different orders.
%     \item Different uniform distributions: $[-1,1]$, $[0,0.5]$, $[0,1]$, and Gaussian.
% \end{itemize}

{\emph{Case of k = 1:}} 
Fig.~\ref{fig1} depicts the largest elements and accuracies of all tested algorithms for random CP tensors with different uniform distributions. 
From Fig.~\ref{fig1}(d), we observe that the proposed algorithm outperforms the baselines in terms of accuracy. Specifically, compared to power iteration, star sampling, and MinCPD, it achieves improvements of $16.0\%\sim108.0\%$, $16.0\%\sim235.7\%$, and $13.2\%\sim372.7\%$, respectively. Although the performance of all tested algorithms is sensitive to data distribution, our algorithm remains more reliable, e.g., its accuracy on the $U(-1,1)$ distribution can still exceed $50\%$. Additionally, Fig.~\ref{fig1}(a) and (d) show that the value obtained by star sampling is often far from the exact one, and the MinCPD method is sensitive to the numerical scale of data. In contrast, our algorithm exhibits neither of these issues, demonstrating its superiority in stability.

Since the star sampling and the power iteration methods can not apply to the smallest element retrieval problem, the comparative analysis of Fig.~\ref{fig2} is confined to MinCPD and our algorithm. As shown in Fig.~\ref{fig2}(d), the MinCPD method performs worst across all uniform distributions, with an accuracy of no more than $40\%$, while the accuracy of our algorithm achieves $48.0\%\sim93.0\%$. Simultaneously, Fig.~\ref{fig2}(a) shows that the value obtained by our algorithm is generally smaller than MinCPD, which implies that it is closer to the exact value. 

%%%%%%%%%%%%%%%%%%%%%%%%%%%%%%%%
\begin{figure*}
\centering
\subfigure[]{
\includegraphics[width=0.45\linewidth]{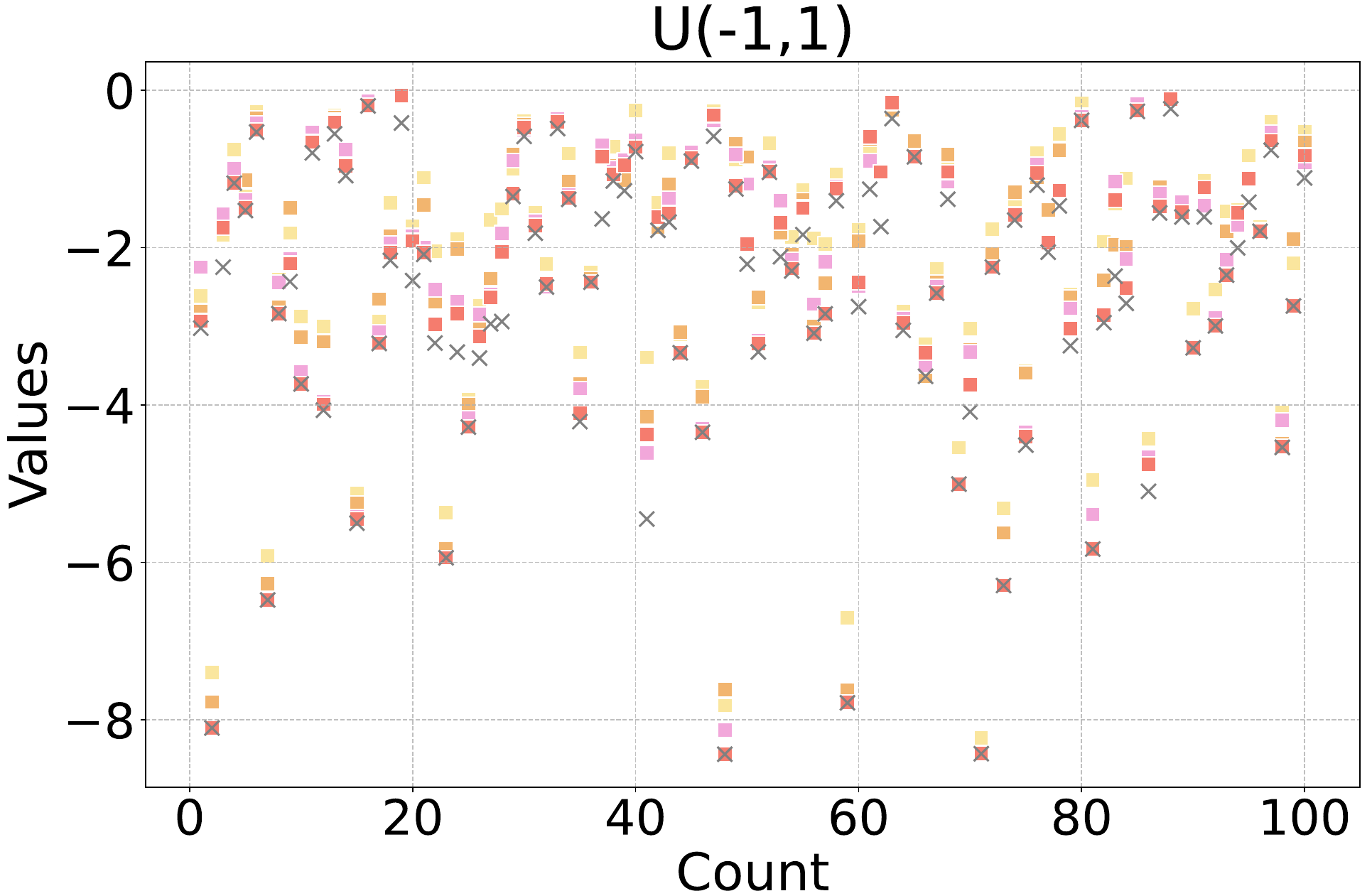}}
\hspace{0.05\linewidth}
\subfigure[]{
\includegraphics[width=0.45\linewidth]{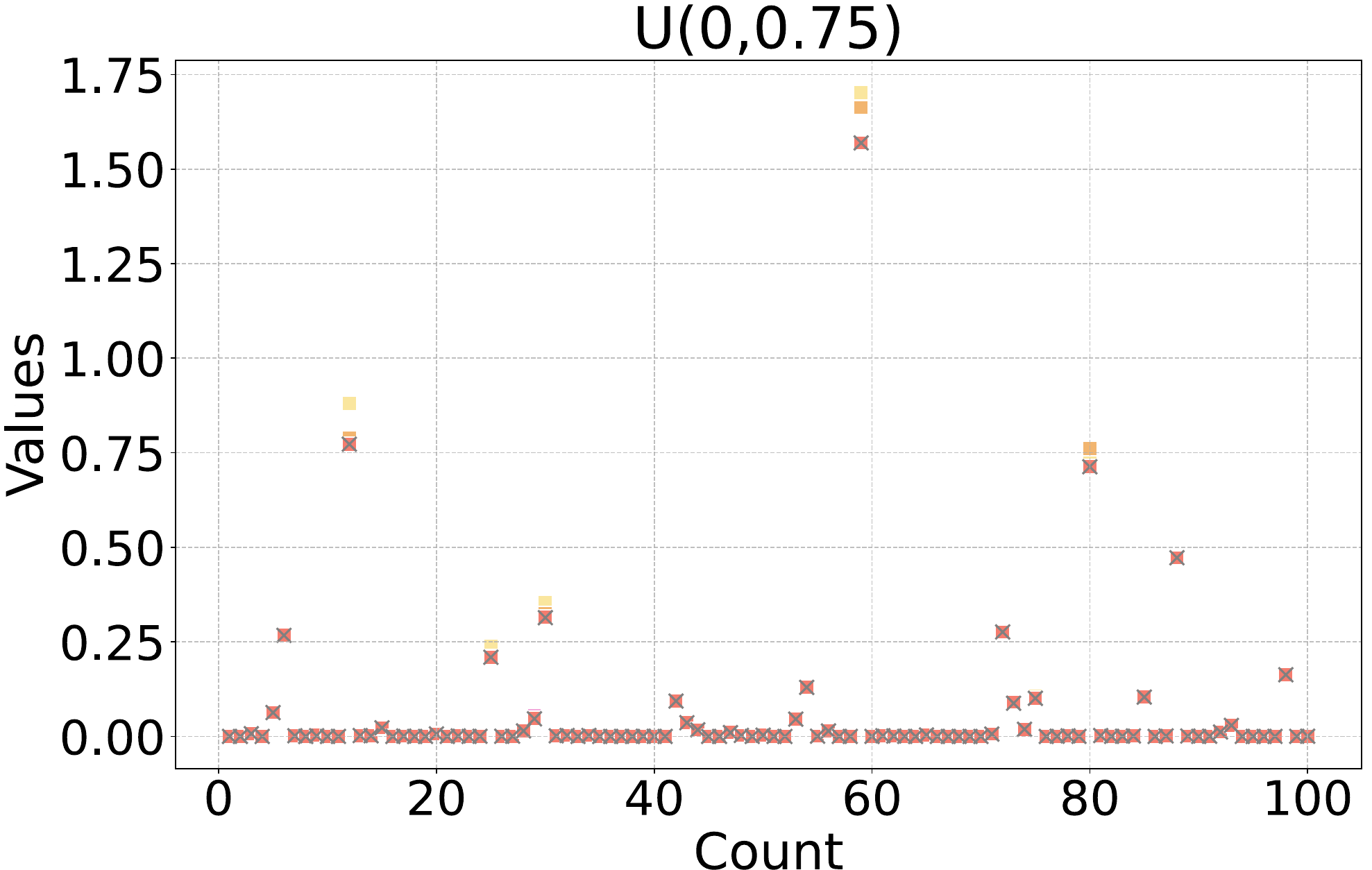}}
\vfill
\subfigure[]{
\includegraphics[width=0.45\linewidth]{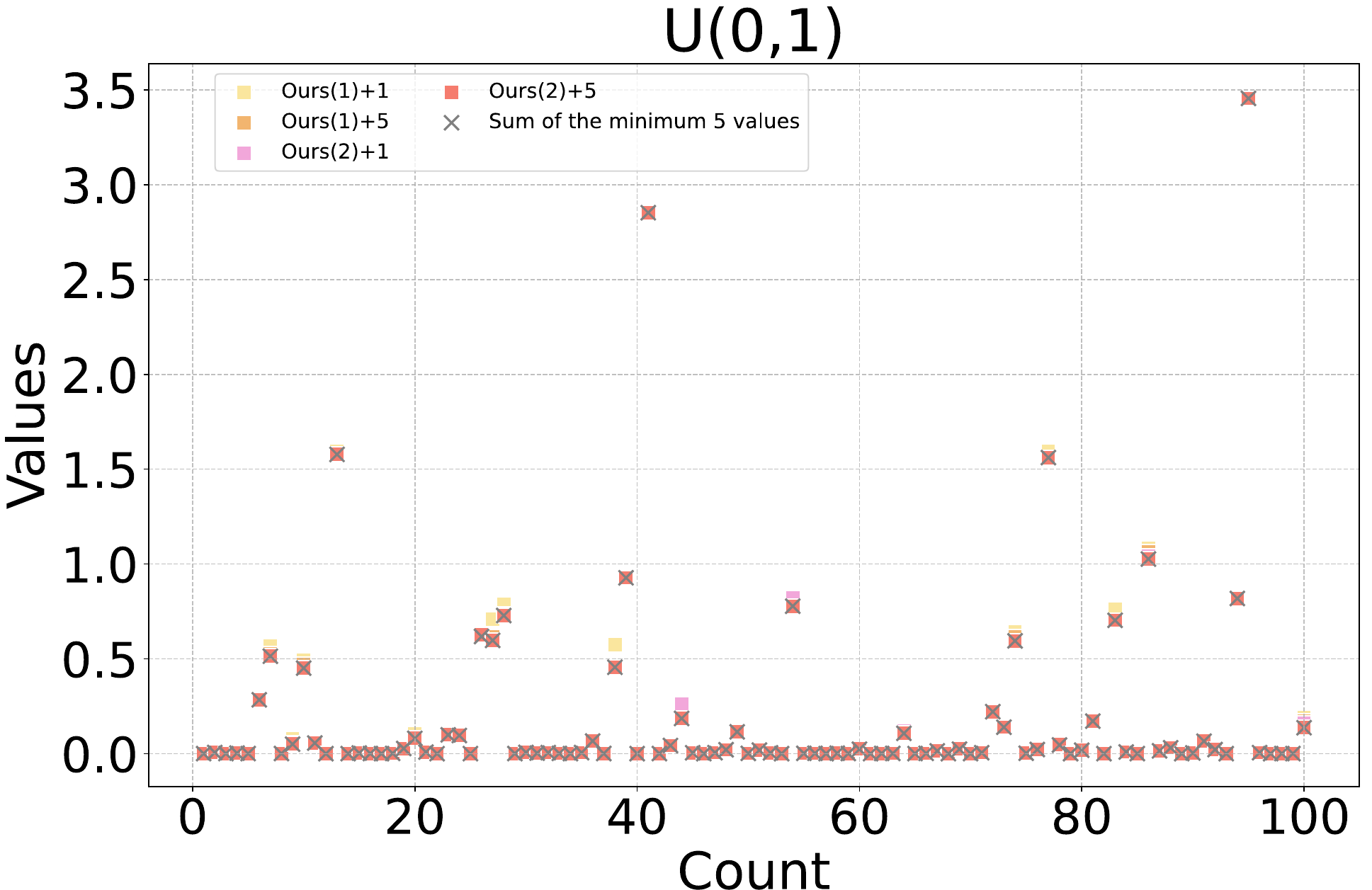}}
\hspace{0.05\linewidth}
\subfigure[]{
\includegraphics[width=0.45\linewidth]{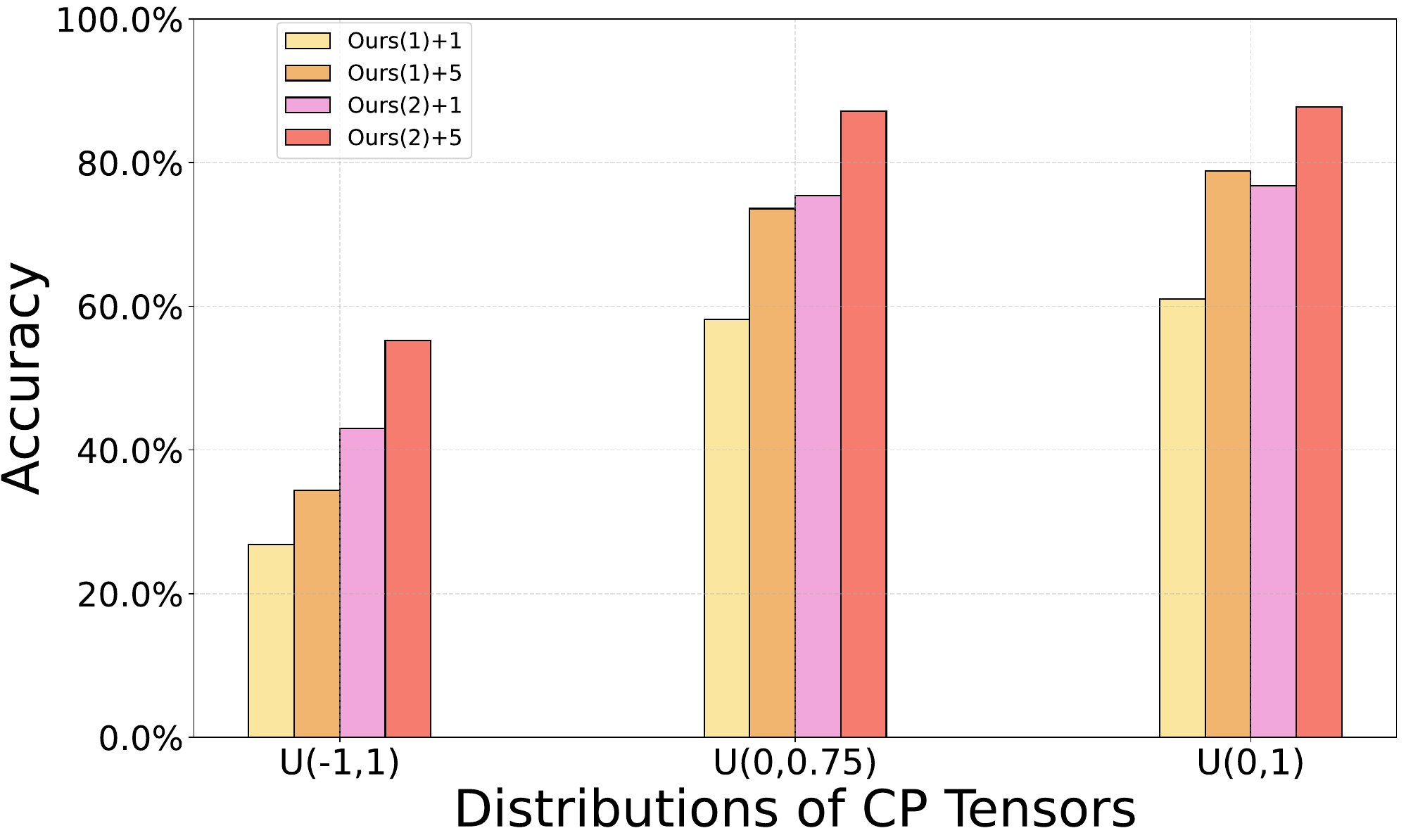}}
\caption{The sum of the 5 smallest values and accuracy obtained by our algorithm for random CP tensors with different uniform distributions.}
\label{fig4}
\end{figure*}

%%%%%%%%%%%%%%%%%%%%%%%%%%%%%%%%%%%%%%%%%%%%

{\emph{Case of k = 5:}} Among all tested algorithms, only the star sampling algorithm and our algorithm are applicable to the $k$ ($k>1$) largest elements retrieval problem. Therefore, Fig.~\ref{fig3} only demonstrates their performance on random CP tensors with different uniform distributions, including the sum of the $5$ largest elements and accuracy. For the $U(-1,1)$ distribution, Fig.~\ref{fig3}(a) shows that the sum of values obtained by star sampling will be smaller than~$0$, indicating excessive deviation from the exact value. Indeed, Fig.~\ref{fig3}(d) illustrates that the star sampling method has a poor performance in terms of accuracy, which remains below $45.0\%$ across all uniform distributions. On the other hand, our algorithm has a higher accuracy. Specifically, for $U(-1,1)$, $U(0,0.75)$, and $U(0,1)$ distributions, its accuracy can reach $59.0\%$, $84.0\%$, and $82.0\%$, respectively. Moreover, our algorithm is also effective at retrieving the $5$ smallest elements. As shown in Fig.~\ref{fig4}(d), it achieves the accuracy of $55.2\%$, $87.2\%$, and $87.8\%$ across these three uniform distributions. This further confirms the advantages of our algorithm in both versatility and reliability.

Notably, the parameters $s$ and $K$ in our algorithm are used to expand the search space for retrieving the top-$k$ elements. The numerical results in this experiment illustrate that adjusting these two parameters significantly improves the accuracy of our algorithm in different scenarios. Particularly, for the $U(-1,1)$ distribution, our algorithm with $s=2$ and $K=5$ (i.e., Ours(2)+5) can improve accuracy by $21.0\%\sim 188.9\%$, which greatly enhances its reliability.

\subsection{Tests on CP tensors from multivariate functions}

In the second example, we examine the performance of the proposed block-alternating iterative algorithm in retrieving the smallest element of CP tensors from two multivariate functions, which include Griewank and Schwefel\tablefootnote{Retrieved from \url{http://www-optima.amp.i.kyoto-u.ac.jp/member/student/hedar/Hedar_files/TestGO.htm}}, i.e., 
% \begin{small}
\begin{equation}\label{eq:test_function}
\begin{array}{l}
      % \text{Griewank function:} \\
      f(\bm{z}) = \sum\limits_{p=1}^d\frac{z_p^2}{4000}-\prod\limits_{p=1}^d\cos(\frac{z_p}{\sqrt{p}})+1,\ z_{p}\in[-600,600],\\
      % \text{Schwefel function:}\\
      f(\bm{z})=418.9829d - \sum\limits_{p=1}^{d}z_{p}\sin(\sqrt{|z_{p}|}),\  z_{p}\in[-500,500],
\end{array}
\end{equation}
% \end{small}
where $d$ is the dimension. In this example, we set the dimension $d$ to $10$. Meanwhile, the grid size of each dimension is a randomly chosen integer between $[2, n]$, where $n$ is taken as $\{128,256,512,1024\}$. Due to the separability of Griewank and Schwefel functions, the CP tensor can be directly derived from Eq.~\eqref{eq:test_function}. 

For each $n$, we generate 100 CP tensors and record the average of the smallest values obtained by the MinCPD algorithm and the proposed algorithms in Table~\ref{table:ex2}. 
As illustrated in Table~\ref{table:ex2}, our algorithm can more accurately identify the smallest element of CP tensors from the Griewank and Schwedel functions. Specifically, for CP tensors from the Griewank and Schwedel functions, the average of values obtained by our algorithm is $0.04\sim14.08$ and $33.22\sim295.43$ smaller than that of MinCPD under different grid sizes. Simultaneously, Fig.~\ref{fig5}(a) shows that the values obtained by the MinCPD method have more outliers, indicating its instability. Additionally, similar to the first experiment, the numerical results in this experiment also demonstrate that our algorithm, with a larger block parameter $s$, can retrieve a more accurate value.

\begin{table*}[t]\label{table:time}
	\normalsize
	\renewcommand\arraystretch{1.35}
	\setlength{\tabcolsep}{12.0pt}
	\begin{center}
		\caption{\normalsize The average of the smallest values obtained by MinCPD and our algorithm for CP tensors from Griewank and Schwefel functions.}\label{table:ex2}
		\scalebox{0.85}{\begin{tabular}{c|c|c|c|c|c}
			\hline\hline
		\multicolumn{2}{c|}{Algorithms} & 128 & 256  & 512 & 1024 \\
  \hline\hline
\multirow{3}{*}{Griewank} & MinCPD & 22.8719 & {2.1701} &  2.9565 & 1.8190  \\
& Ours($1$) & {8.8038} & 2.1828  &  {2.9493}  & {1.8020}   \\
  %\cline{2-5}
 & Ours($2$) & {8.7888} & {2.1280}&  {2.8561} & {1.6819} \\
  %\cline{2-5}
 \hline\hline
 \multirow{3}{*}{Schwefel} & MinCPD & 507.4376 &242.5032  &  75.6010 &  178.0408 \\
 & Ours($1$) & {212.0115} & {102.6022} & {42.3849}  & {36.2540}  \\
 % \cline{2-5}
 & Ours($2$) & {212.0115} & {102.6022} & {42.3849} & {36.2540}  \\
  %\cline{2-5}
			\hline\hline
		\end{tabular}}
	\end{center}
\end{table*}

%%%%%%%%%%%%%%%%%%%%%%%%%%%%%%%%%%%%%%%%%%%%%%%%

\begin{figure*}
\centering
\subfigure[]{
\includegraphics[width=0.45\linewidth]{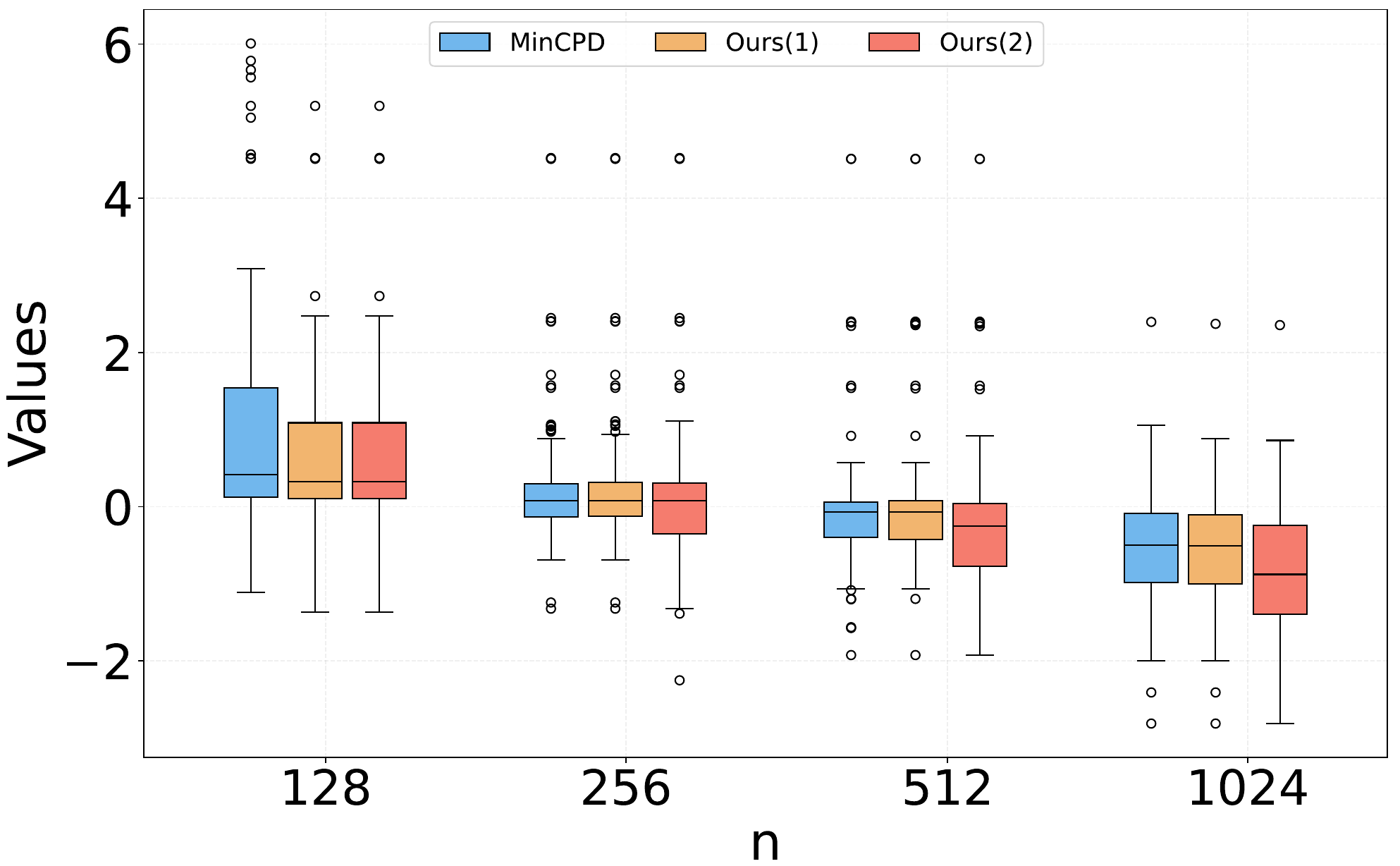}}
\hspace{0.05\linewidth}
\subfigure[]{
\includegraphics[width=0.45\linewidth]{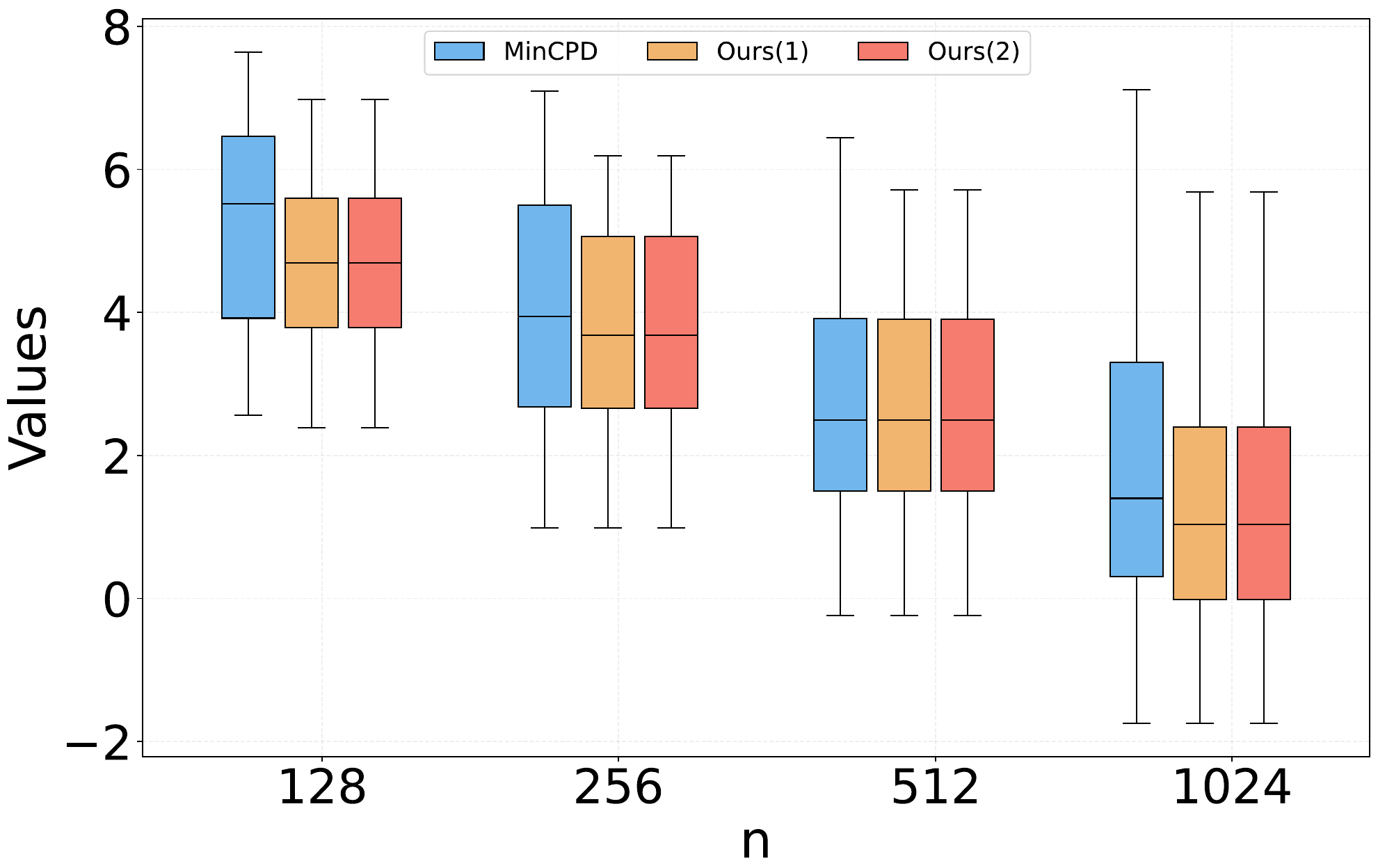}}
\caption{The smallest value obtained by MinCPD and our algorithms for CP tensors from (a) Griewank and (b) Schwefel functions.}
\label{fig5}
\end{figure*}

%%%%%%%%%%%%%%%%%%%%%%%%%%%%%%%%%%%%%%%%%%%%%%%%

\subsection{Tests on quantum circuit simulations}

The goal of quantum circuit simulations is to calculate the distribution of quantum devices equipped with the circuit on a classical computer, i.e., the amplitudes of the bitstrings of the final quantum state from a quantum circuit mapping an initial quantum state. This is of great significance in the era of noisy intermediate-scale quantum (NISQ). On the one hand, it denotes what is the boundary of quantum supremacy. Simultaneously, it provides guidance for quantum computers, which is valuable for the development of quantum correction and the verification of quantum systems \cite{preskill2018quantum}. 

This example focuses on the quantum Fourier transform (QFT) circuit \cite{coppersmith2002an}, as shown in Fig.~\ref{quantum}. We present a subspace-CP tensor model to represent the quantum state that lies in the full Hilbert space $(\mathbb{C}^{2})^{\otimes d}$. Unlike the naive CP tensor model proposed in \cite{ma2022low}, we rearrange the quantum state as a $p$th-order tensor $\hat{\tensor{A}}\in\mathbb{C}^{2^q\times2^q\times\cdots\times2^q}$ (i.e., $d=pq$) instead of $d$th-order tensor, and use the CP format to represent it. The quantum circuit simulation consists of two parts: the quantum gate applying process and the measurement process. For the first part, the quantum gates can be implemented by tensor-times-matrix (TTM) operations on CP tensors \cite{kolda2009tensor}. The second part is to obtain the standard quantum state corresponding to the largest amplitude, which is equivalent to the top-$k$ largest elements retrieval problem considered in this paper. To compare the proposed methods with baselines, we take the number of qubits $d$ as $l^2$ with $l\in\{3,4,5,6,7\}$, and set $p=q=l$. The initial quantum state is set to a randomly generated rank-one tensor, where each element of each CP factor is a complex number with both real and imaginary parts drawn from the $U(0,1)$ distribution.
%The initial quantum state is set to a randomly generated rank-one tensor, and each element of each CP factor is a complex number whose real and imaginary parts are both drawn from the $U(0,1)$ distribution. 
Additionally, we take $k$ as~$1$ and~$5$, respectively. The average of the $k$ largest elements obtained by all tested algorithms is shown in Table~\ref{table:qft}, where ``--'' indicates that the corresponding tested method does not support this feature. Fig.~\ref{fig:fft} displays the corresponding accuracy of these tested algorithms.

%%%%%%%%%%%%%%%%%%%%%%%%%%%%%%%%%%%%%%%%%%%%
\begin{figure}[htbp]
\centering
\includegraphics[scale=0.28]{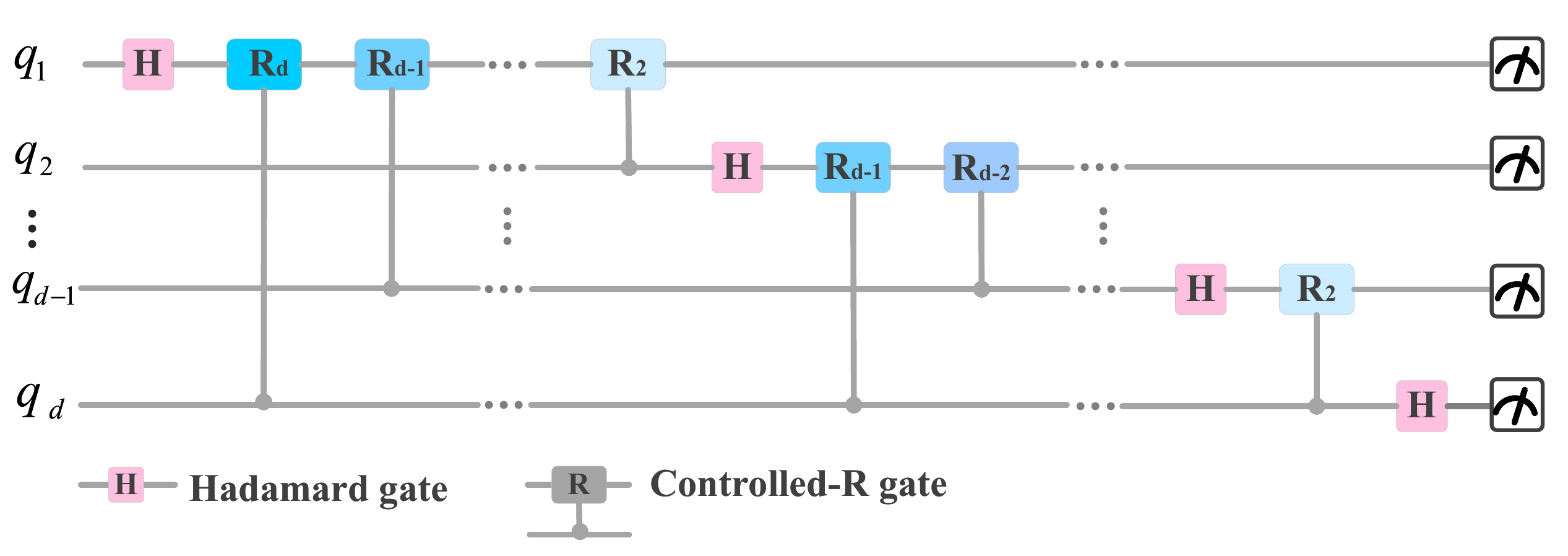}
\caption{The description of the quantum Fourier transformation circuit with $d$ qubits.}
\label{quantum}
\end{figure}

%%%%%%%%%%%%%%%%%%%%%%%%%%%%%%%%%%%%%%%%%%%%

From Table~\ref{table:qft}, it can be seen that the value obtained by the power iteration method is far from the largest value, especially when the number of qubits is large. This is because the accuracy of the power iteration method is highly dependent on the data distribution. As the number of qubits increases, each element of the CP tensor from the quantum state gradually decreases, causing the error generated during iteration to become the dominant term and leading to the inaccuracy of the power iteration method. Compared to the baselines, our algorithm performs better on CP tensors from quantum states with different numbers of qubits. For the numbers of qubits of $9$, $16$, and $25$, our algorithm can accurately retrieve the largest element. As shown in Fig.~\ref{fig:fft}, its accuracy in retrieving the $5$ largest elements reaches $100.0\%$, $90.6\%$, and $90.2\%$, respectively. Due to limited memory, we are unable to verify the accuracy of our algorithm when the number of qubits is greater than $36$. However, Table~\ref{table:qft} shows that our algorithm yields larger values than baselines for both cases of $k=1$ and $k=5$, indicating greater accuracy. It is worth mentioning that the accuracy of baselines decreases with the increase of the number of qubits, whereas our algorithm does not exhibit this phenomenon. This illustrates that our algorithm is more stable and capable of handling larger-scale problems.

\begin{table}[htp]
	\scriptsize
	\renewcommand\arraystretch{1.35}
	\setlength{\tabcolsep}{5.5pt}
	\begin{center}
		\caption{\small The sum of the $k$ ($k=1$ and $5$) largest values obtained by baselines and our algorithm for CP tensors from QFT simulation.}\label{table:qft}
		\scalebox{0.85}{\begin{tabular}{c|c|c|c|c|c|c|c|c|c|c}
			\hline\hline
		\multicolumn{1}{c|}{Number of qubits} & \multicolumn{2}{c|}{9} & \multicolumn{2}{c|}{16} & \multicolumn{2}{c|}{25}  & \multicolumn{2}{c|}{36} & \multicolumn{2}{c}{49}  \\
  \hline\hline
  CP rank & \multicolumn{2}{c|}{8} & \multicolumn{2}{c|}{20} & \multicolumn{2}{c|}{56} & \multicolumn{2}{c|}{61} & \multicolumn{2}{c}{297}  \\
  \hline\hline
  Values & 0.0708 & 0.2777 & 5.8018e-3 & 0.0237 & 5.6542e-4 & 1.8245e-3 & \multicolumn{4}{c}{Out of memory} \\
  \hline\hline
  Power iteration & 1.0784e-5 & -- &9.3612e-5 & -- &6.1238e-7 & --& 1.5432e-10 &-- &1.1946e-12 & --\\
  Star sampling+1 & 0.0351 & 0.1477 & 5.8018e-3& 0.0236 & 5.6542e-4&1.7474e-3 &1.2589e-4 &5.2523e-4 & 5.3105e-8 &  2.2199e-7\\
  Star sampling+5 & 0.0351 & 0.1477 &5.8018e-3 & 0.0236 & 5.6542e-4& 1.7474e-3& 1.2589e-4 &5.2523e-4 & 5.3105e-8& 2.2199e-7 \\
  MinCPD & 0.0708 & -- &5.8018e-3 &-- & 1.5236e-4 &-- & 1.2693e-4 &-- & 8.0451e-8&-- \\
  Ours(1)+1 & 0.0708 & 0.2208 &5.8018e-3 & 0.0194&5.6542e-4 &1.6955e-3 & 1.2693e-4 & 4.7627e-4 &2.1279e-7 & 8.5845e-7 \\
  Ours(1)+5 &0.0708 & 0.2271 &5.8018e-3 & 0.0211&5.6542e-4 &1.7780e-3  &1.2693e-4  & 5.1209e-4 &2.1279e-7 & 9.9251e-7 \\
  Ours(2)+1 &0.0708 & 0.2271 &5.8018e-3 &0.0208&5.6542e-4 &1.7971e-3 &1.2693e-4  & 5.0967e-4 &2.1279e-7 & 9.7743e-7\\
  Ours(2)+5 &0.0708 & 0.2277 &5.8018e-3 &0.0230 &5.6542e-4 &1.8207e-3 &1.2693e-4  &5.2817e-4 &2.1279e-7 & 9.9732e-7 \\
  			\hline\hline
		\end{tabular}}
	\end{center}
\end{table}

%%%%%%%%%%%%%%%%%%%%%%%%%%%%%%%%%%%%
\begin{figure}[t]
\centering
\vspace{0.1cm}
\setlength{\abovecaptionskip}{0.1cm}  
\setlength{\belowcaptionskip}{0.1cm}  
\includegraphics[width=0.45\textwidth]{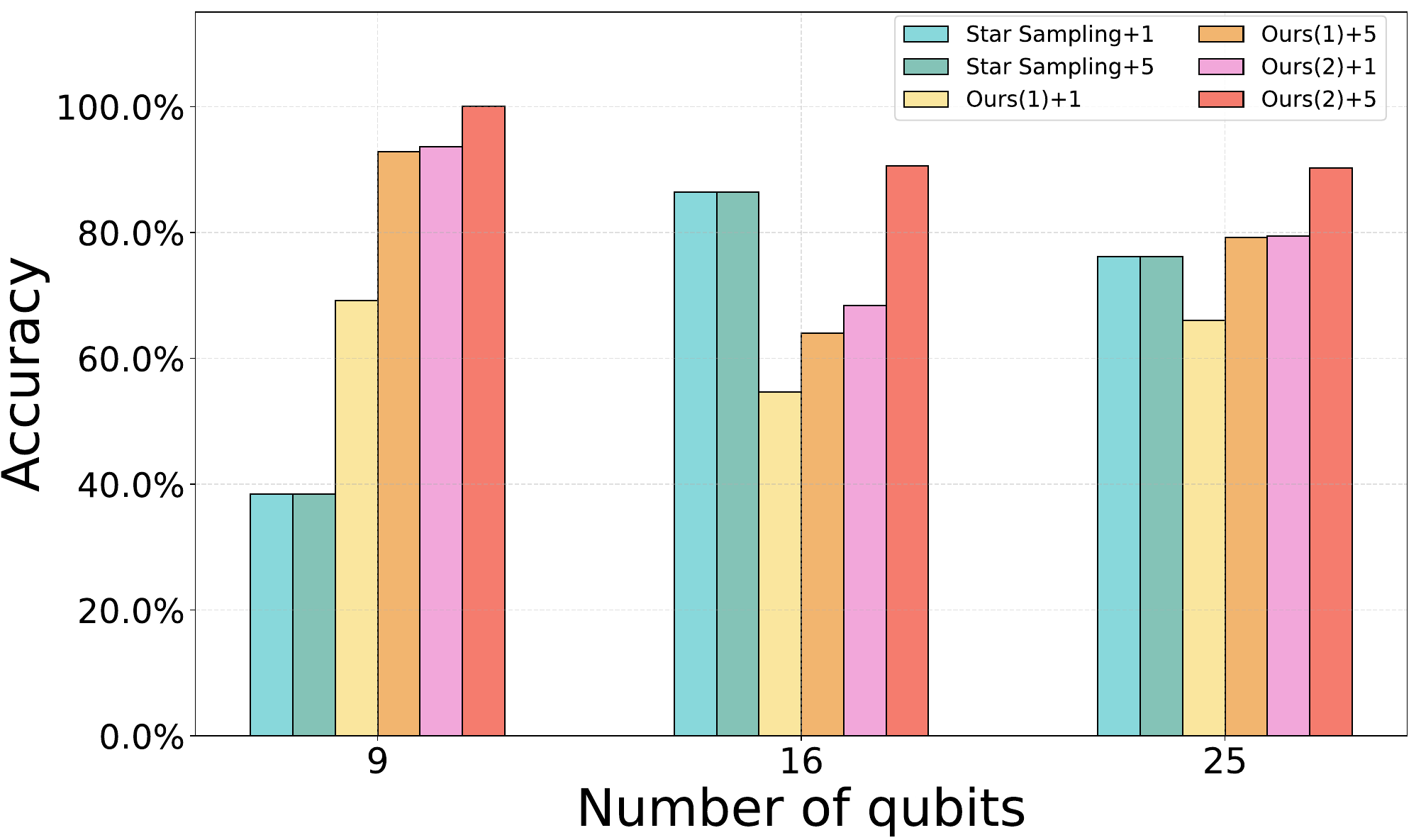}
\centering
\caption{The accuracy obtained by baselines and our algorithm for CP tensors from QFT simulation.}
\label{fig:fft}
\end{figure}
%%%%%%%%%%%%%%%%%%%%%%%%%%%%%%%%%%%%

\section{Conclusion}\label{sec:con}

This paper proposes a block-alternating iterative algorithm for efficiently retrieving the top-$k$ elements in a factorized tensor.
Building upon the previously proposed symmetric eigenvalue model, 
we first model the top-$k$ elements retrieval problem into a continuous constrained optimization problem via the rank-one structure of the tensor corresponding to the eigenvector. Subsequently, we design a block-alternating iterative algorithm to solve the equivalent continuous constrained optimization problem. For small-scale subproblems in the iterative process, a heuristic method is proposed with the separable summation form of the objective function. Numerical experiments with synthetic and real-world tensors demonstrate the effectiveness of our proposed algorithm in various scenarios. Possible future work may include applying our proposed block-alternating iterative algorithm to other practical applications, such as recommendation system, network measurement, and computational biology.

\bibliography{iclr2024_conference}
\bibliographystyle{iclr2024_conference}

% \subsection{Tests on random tensors with uniform distribution}

% \begin{figure*}[htb]
%     \centering
%     \subfigure[For third-order tensors. ]{
%     \rotatebox{90}{~~~~ Maximum Value}
%      \begin{minipage}[b]{0.3\linewidth}
%       \includegraphics[width=1\linewidth]{figures/max-order3.pdf}
%      \end{minipage}
%      \label{fig:3-order}
%     }
%      \subfigure[For fourth-order tensors.]{
%      \begin{minipage}[b]{0.3\linewidth}
%       \includegraphics[width=1\linewidth]{figures/max-order4.pdf}
%      \end{minipage}
%       \label{fig:4-order}
%     }
%     \subfigure[For fifth-order tensors. ]{
%      \begin{minipage}[b]{0.3\linewidth}
%       \includegraphics[width=1\linewidth]{figures/max-order5.pdf}
%      \end{minipage}
%      \label{fig:5-order}
%     }
%     \subfigure[For third-order tensors. ]{
%      \rotatebox{90}{~~~~ Minimum Value}
%      \begin{minipage}[b]{0.3\linewidth}
%       \includegraphics[width=1\linewidth]{figures/min-order3.pdf}
%      \end{minipage}
%      \label{fig:3-order}
%     }
%      \subfigure[For fourth-order tensors.]{
%      \begin{minipage}[b]{0.3\linewidth}
%       \includegraphics[width=1\linewidth]{figures/min-order4.pdf}
%      \end{minipage}
%       \label{fig:4-order}
%     }
%      \subfigure[For fifth-order tensors. ]{
%      \begin{minipage}[b]{0.3\linewidth}
%       \includegraphics[width=1\linewidth]{figures/min-order5.pdf}
%      \end{minipage}
%      \label{fig:5-order}
%     }
%     \caption{Accuracy on randomly generated tensors with $[0,1]$ uniform distribution.}
%     \label{fig:random-accuracy}
% \end{figure*}

\end{document}